\documentclass[12pt,letter]{amsart}
\usepackage[latin1]{inputenc}
\usepackage{epsfig}
\usepackage{epic}
\usepackage{eepic}
\usepackage{threeparttable}
\usepackage{amscd}
\usepackage{here}
\usepackage{graphicx}
\usepackage{lscape}
\usepackage{tabularx}
\usepackage{subfigure}
\usepackage{longtable}
\usepackage{amsmath,amsthm,amssymb,latexsym,stackrel}
\providecommand{\dotdiv}{
  \mathbin{
    \vphantom{+}
    \text{
      \mathsurround=0pt 
      \ooalign{
        \noalign{\kern-.35ex}
        \hidewidth$\smash{\cdot}$\hidewidth\cr 
        \noalign{\kern.35ex}
        $-$\cr 
      }%
    }%
  }%
}
\usepackage[T1]{fontenc}
\usepackage{libertine} 
\usepackage{palatino} 
\usepackage{euler} 
\usepackage{eucal} 
\usepackage{graphicx}
\usepackage{pb-diagram}
\usepackage{pst-3d,pst-coil,pst-eps,pst-fill,pst-grad,pst-math,pst-text,pst-tree}
\usepackage[all,matrix,arrow]{xy}                              
\newcommand\dotminus{%
  \ooalign{\hidewidth\raise1ex\hbox{.}\hidewidth\cr$\;-\;$\cr}%
}

\newtheorem{proposition}{Proposici\'on}[section]

\newtheorem{assumption}[proposition]{Assumption}

\newtheorem{Prop}[proposition]{Proposition}
\newtheorem{Coro}[proposition]{Corollary}

\newtheorem{lem}[proposition]{Lemma}
\newtheorem{hec}[proposition]{Fact}
\newtheorem{theo}[proposition]{Theorem}
\newtheorem{nota}[proposition]{Note}
\newtheorem{defi}[proposition]{Definition}
\newtheorem{Ej}[proposition]{Example}
\newtheorem{Ejs}[proposition]{Examples}

\newtheorem{remark}[proposition]{Remark}
\newtheorem{notation}[proposition]{Notation}

\newcommand{\eop}[1]{
\hspace{10mm} \vspace{-6mm}
\begin{flushright}
\qedsymbol$_{\text{#1}}$\\ \ \\
\end{flushright}
}
\newenvironment{prueba}[1][{\it Proof}]{\noindent {\it #1.} }{}
\newcommand{\bdem}[1][Proof]{\begin{prueba}[#1]}
\newcommand{\edem}[1][]{\eop{#1}
\end{prueba}}

\def\bsdem{\begin{prueba}[Reference]}\def\bsindem{\begin{proof}[\ ]}

\mathchardef\mhyphen="2D

 \author[D. Reyes]{David Reyes}
\email{davreyesgao@unal.edu.co}
\address{Departamento de Matem\'aticas, Universidad Nacional de Colombia, AK 30 $\#$ 45-03 c\'odigo postal 111321, Bogota, Colombia.}

\author[P. Zambrano]{Pedro H. Zambrano}
\email{phzambranor@unal.edu.co, phzambranor@gmail.com}
\address{Departamento de Matem\'aticas, Universidad Nacional de Colombia, AK 30 $\#$ 45-03 c\'odigo postal 111321, Bogota, Colombia.}

\thanks{The first author wants to thank the second author for the time devoted to advise his undergraduate thesis, where this paper is one of its fruits. The second author wants to thank Universidad Nacional de Colombia for the grant ``Convocatoria para el apoyo a proyectos de investigaci\'on y creaci\'on art\'istica de la sede Bogot\'a de la Universidad Nacional de Colombia - 2019'' c\'odigo Hermes 48359}

\begin{document}
\title[A characterization of Continuous Logic $\dots$]{A characterization of Continuous Logic by using quantale-valued logics}

\date{}
\begin{abstract}
In this paper, we propose a generalization of Continuous Logic (\cite{BeBeHeUs}) where the distances take values in suitable co-quantales (in the way as it was proposed in \cite{Fla97}). By assuming suitable conditions (e.g., being co-divisible, co-Girard and a $\mathbf{V}$-domain), we provide, as test questions, a proof of a version of the Tarski-Vaught test (Proposition~\ref{test de Tarski-Vaught}) and \L o\'s Theorem (Theorem~\ref{Teorema de Los}) in our setting. Iovino proved in \cite{Io01} that there is no any logic extending (equivalent logics to) Continuous Logic satisfying both Countable Tarski-Vaught chain Theorem and Compactness Theorem. Since $[0,1]$ satisfies all of the assumptions given above, we get new logics by dropping any of those assumptions.

\smallskip
\noindent \textbf{Keywords.} {metric structures, lattice valued logics, co-quantales, co-Girard, co-divisibility, domains, Tarski-Vaught test, \L o\'s theorem}.

\smallskip
\noindent \textbf{Mathematics Subject Classification 2020.} {Primary: 03C90, 03B60, 03C65, 03C20, 03C50.
Secundary:  06F07, 18B35, 03C66.}

\end{abstract}
\date{\today}
\maketitle

\section{Introduction}

S. Shelah and J. Stern proved in~\cite{ShSt78} that a first order attempt of study of classes of Banach Spaces has a ``bad'' behavior  (this has a very high Hanf number, having a model-theoretical behavior similar to a second order logic of binary relations). This led to develop a suitable logic beyond first order logic, in order to do a suitable model-theoretic analysis of Banach Spaces.
\\ \\
\indent In \cite{ChKe66}, C. C. Chang and H. J. Keisler proposed a new logic with truth values within a compact Hausdorff topological space, which was the first time where the term {\it Continuous Logic} appeared. They developed basics on Mathematical Logic in this book, but by then Model Theory had not been very developed (Morley's first order categoricity theorem had just been proved, and there was no stability theory by then). For some reason, people did not continue working on this kind of logic until it was rediscovered in the 90's by W. Henson and J. Iovino (see~\cite{HeIo02,Io95}) and later by I. Ben-Yaacov et al (see~\cite{BeBeHeUs}), but in the particular case by taking the truth values in the unitary interval $[0,1]$, focusing on the study of structures based on complete metric spaces (e.g. Hilbert spaces together with bounded operators -see~\cite{ArBe09}-, Banach spaces, Probability spaces -see \cite{BeHe04}-). This logic is known as {\it Continuous Logic} Because of some technical reasons, people working on Continuous Logic have to consider strong assumptions on the involved operators  (e.g. boundness) in order to axiomatize classes of metric structures in this logic. This took us to the notion of {\it Metric Abstract Elementary Class} (see \cite{HiHy09,Za11}) for being able to deal with non-axiomatizable -in Continuous Logic- classes of complete metric structures. However, this approach does not consider topological spaces in general.
\\ \\
\indent 
Independently, Lawvere provided in~\cite{Law73} a framework in Category Theory for being able to consider a logic with generalized truth values in order to study metric spaces from this point of view. However, there is no a deep model-theoretic study in this paper. 
\\ \\
\indent 
There is an attempt of a study of first order Model Theory for Topological Spaces (see~\cite{FlZi}), but it was just suitable to study particular algebraic examples like Modules and Topological Groups, due to the algebraic nature of first order logic. Moreover, this approach was left as a model-theoretic study of general Topological Spaces but it was the beginning of the model-theoretic study of Modules (see~\cite{Prest}).
\\ \\
\indent 
Quantales are a suitable kind of lattices introduced for being able to deal with locales (a kind of lattices which generalizes the ideal of open sets of a topological space) and multiplicative lattices of ideals from Ring Theory and Functional Analysis (e.g., $C^*$-algebras and von Neumann algebras).
\\ \\
\indent 
Considering the contravariant notions in a quantale (which we will called {\it co-quantales}), Flagg gave in~\cite{Fla97} (Theorem 4.15) a way to deal with topological spaces as pseudo metric spaces (the Flagg's notion called {\it continuity spaces}) where the distance takes values in a suitable quantale built from the topology.
\\ \\
\indent 
In this paper, we will propose a generalization of Continuous Logic by defining distances with values in value co-quantales together with suitable assumptions (e.g., being co-divisible (or best, substractibility), co-Girard and a $\mathbf{V}$-domain). Potentially, we might study, from this point of view, topological structures more general than real-valued metric structures.
\\ \\
\indent 
In \cite{Io01}, J. Iovino proved that there is no logic for analytic structures that extends properly (logics equivalent to) Continuous Logic and
satisfies both the Compactness Theorem and the Elementary Chain Property. In this paper, by assuming a value co-quantale to be co-Girard (see Definition~\ref{cuantal co-Girard}) we prove that the related logic satisfies the Tarski-Vaught test (Proposition~\ref{test de Tarski-Vaught}) and therefore the Tarski-Vaught chain theorem holds.
\\ \\
\indent 
Additionally, if we assume that a value co-quantale is co-divisible (substractable) and a V-domain, then the related logic satisfies a version of the \L o\'s Theorem (Theorem~\ref{Teorema de Los}) and hence it satisfies compactness.
\\ \\
\indent 
On the other hand, $\mathbb{V}:=[0,1]$ satisfies to be co-Girard, co-divisible (substractable) and a $V$-domain, then if we drop any of these assumptions we get new logics. 
\\ \\
\indent 
This paper is organized as follows: In the second section we will provide basic facts in co-quantales. In the third section, we introduced our approach to co-quantales valued logics, analogously as it is done in Continuous Logic but by considering distances with values in a suitable co-quantale. As test questions, we provide a proof of a version of the Tarski-Vaught test (Proposition~\ref{test de Tarski-Vaught}) and a version of \L o\'s Theorem (Theorem~\ref{Teorema de Los}). A difference between our approach and Continuous Logic lies on the fact that we can provided a version of \L o\'s Theorem for $D$-products (before doing the quotient to force a $D$-product being an actual metric space -which is called a $D$-ultraproduct in Continuous Logic-). We notice that the same proof for $D$-products works for $D$-ultraproducts in our setting. As consequences of \L o\'s Theorem, in a similar way as in first order and Continuous logics, we provide a proof of a version of Compactness Theorem (Corollary~\ref{teorema de compacidad}) and of the existence of $\omega_1$-saturated models (Proposition~\ref{existencia de modelos saturados}).

\section{Value co-quantales and Continuity Spaces}
In this section, we provide some basics on value co-quantales. 

\subsection{Value lattices}

\begin{defi}\label{relaci\'on prec}
Given $L$ a complete lattice and $x,y\in L$, we say that {\bf $x$ is co-well below $y$} (denoted by $x\prec y$), if and only, if for all $A\subseteq X$, if $\bigwedge A \leq x$ then there exists $a\in A$ such that $a\leq y$.
\end{defi}

\begin{remark}
In $[0,\infty]$, $x\prec y$ agrees with $x<y$.
\end{remark}


\begin{lem} (\cite{Fla97}; Lemma 1.2)\label{propiedades b\'asicas de la relaci\'on prec}
Let $L$ be a complete lattice, then for all $x,y,z \in L$ we have that
\begin{enumerate}
    \item $y\prec x$ implies $y\leq x$
    \item  $z\leq y$ and $y\prec x$ imply $z\prec x$
    \item  $y\prec x$ and $x\leq z$ imply $y\prec z$.
\end{enumerate}
\end{lem}

\begin{lem}\label{continuidad} (\cite{Fla97}; Lemma 1.3)
If $L$ is a complete lattice, then for all $A\subseteq L$ and $x\in L$ we have that $\bigwedge A  \prec x$ if and only if there exists $a\in A$ such that $a\prec x$.
\end{lem}

\begin{defi} \label{aproximaci\'on}
A complete lattice $L$ is said to be {\bf co-continuous} if  $a=\bigwedge\{b\in L : a\prec b\}$ provided that $a\in L$.
\end{defi}

\begin{remark}
The notion of co-continuity is called {\it completely distributivity} in~\cite{Fla97}, but that notion corresponds to the distributivity of arbitrary joins over arbitrary meets in the standard study of lattices. Because of that, it is a bit confussing following~\cite{Fla97,FlaKop97}. 
\end{remark}


\begin{lem}\label{densidad}(\cite{Fla97}; Lema 1.6])
Given $L$ a co-continuous lattice and $x, y \in L$ provided that $x\prec y$,  there exists $z\in L$ such that $x\prec z$ y $z\prec y$.
\end{lem}

\begin{defi}\label{propiedadefiltro}
A {\bf value lattice} is a co-continuous lattice $L$ provided that 
\begin{enumerate}
    \item $0\prec 1$
    \item if $\delta, \delta^{\prime} \in L$ satisfy $0\prec \delta$ and $0\prec \delta^{\prime}$, then $0\prec \delta \wedge \delta^{\prime} $.
\end{enumerate}
\end{defi}

\begin{Ejs}
\begin{enumerate}
\item The 2-valued Boolean algebra ${\bf 2}:=\{0,1\}$, where $0<1$.
\item The ordinal number $\omega+1:=\{0,1,...,\omega\}$ together with the usual ordering.
\item The unit interval $I=[0,1]$ with the usual real ordering.
\item $([0,1],\ge)=([0,1],\le)^{op}$ (which we will denote by $\mathcal{E}$).
\item $([0,\infty],\le_{usual})$.
\item (\cite{FlaKop97}; pg 115-117)
Given a set $R$, let us denote $\mathcal{P}_{fin}(R)=\{X\in \mathcal{P}(X) : X  \text{ is finite} \}$ and for all $X\in \mathcal{P}_{fin}(R)$ we denote $\downarrow(X)=\{Y\in \mathcal{P}(X) : Y\subseteq X\}$. \\
Given $\Omega(R)=\{p\in \mathbf{P}(P_{fin}(R)) :  X\in p \text{ implies } \downarrow(X)\subseteq p  \}$, then $(\Omega(R),\supseteq)$ is a valued lattice.
\end{enumerate}
\end{Ejs}

\subsection{Value co-quantales}
In this paper, we do not work with the usual notion of quantale. We consider the contravariant notion (which we call {\it co-quantale}) because this approach allows us to work with a notion of distance (pseudo-metric) with values in the co-quantale (see~\cite{Fla97}), in an analogous way as the metric structures given in Continuous Logic. We know that this is not the standard way to study quantales, but we chose this setting in order to do a similar study as it is done in Continuous Logic.

\begin{defi}\label{cuantal}\index{Cuantal}
A {\bf co-quantale} ${\bf V}$ is a complete lattice provided with a commutative monoid structure $({\bf V},+)$ such that
\begin{enumerate}
    \item  The minimum element $0$ of ${\bf V}$ is the identity of $({\bf V},+)$; i.e., $a+0=a$ for all $a\in V$ and
    \item  for all $a\in V$ and $(b_i)_{i\in I}\in V^I$, $a+\bigwedge_{i\in I}b_i=\bigwedge_{i\in I}(a+b_i)$
    
\end{enumerate}
\end{defi}

\begin{Prop}\label{monotonia de la suma}(\cite{Fla97}; Pg 6])
If ${\bf V}$ is a co-quantale, then for all $a, b,c\in {\bf V}$ we have that
\begin{enumerate}
    \item $a+1=1$
    \item $a\leq b$ implies $c+a\leq c+b$
\end{enumerate}
\end{Prop}

\begin{Prop}\label{adjunci\'on}(\cite{Fla97}; Thrm 2.2])
 Given ${\bf V}$ a co-quantale and $a,b\in {\bf V}$, define $a\dotdiv b:=\bigwedge\{r\in V : r + b \geq a\}$. Therefore,  for all $c\in {\bf V}$ we have that
\begin{enumerate}
    \item $a\dotdiv b\leq c$ if and only if $a\leq b + c$
     \item $a\leq (a\dotdiv b)+b$
    \item $(a+b)\dotdiv b\leq a$
     \item $a\dotdiv b=0$ if and only if $a\leq b$
    \item $a\dotdiv(b+c)=(a\dotdiv b)\dotdiv c=(a\dotdiv c)\dotdiv b$
    \item $a\dotdiv c\leq(a\dotdiv b)+(b\dotdiv c)$
\end{enumerate}
\end{Prop}

Since $\dotdiv$ is the left adjoint of $+$ and it preserves categorical limits, we have the following fact.

\begin{hec}\label{resta truncada es adjunto izquierdo}
Let ${\bf V}$ be a co-quantale, then for any sequence $(b_i)_{i\in I}$ and any element $a\in {\bf V}$ we have that $(\bigvee_{i\in I}b_i)\dotdiv a=\bigvee_{i\in I}(b_i\dotdiv a)$ 
\end{hec}

\begin{lem}\label{monoton\'ia-intercambio}
Given a co-quantale ${\bf V}$, for any $a, b\in {\bf V}$, $a\leq b$ implies that for all $c$ we have that $c\dotdiv a\geq c\dotdiv b$ and $a\dotdiv c \leq b\dotdiv c$.
\end{lem}
\begin{proof}
Let $c\in {\bf V}$, then by Proposition~\ref{adjunci\'on} (2) we may say that $c\leq (c\dotdiv a) + a$. So, $c\leq (c\dotdiv a) + b$ whenever $a\leq b$ (by Proposition~\ref{monotonia de la suma} (2)), and so by Proposition~\ref{adjunci\'on} (1) this is equivalent to $c\dotdiv b\leq c\dotdiv a$.\\
On the other hand, since $a\leq b$ then
\begin{eqnarray*}
b&\leq& (b\dotdiv c)+c \ (\text{Proposition~\ref{adjunci\'on} (2)})\\
a&\leq& (b\dotdiv c)+c \ (\text{ since } a\leq b)\\
a\dotdiv c&\leq& b\dotdiv c \ (\text{Proposition~\ref{adjunci\'on} (1)})
\end{eqnarray*}
\end{proof}

\begin{Prop}\label{residuar un \'infimo}
Given ${\bf V}$ a co-quantale, $a\in {\bf V}$ and a sequence $(b_i)_{i\in I}$ in ${\bf V}$, we have that $a\dotdiv\bigwedge_{i\in I}b_i=\bigvee_{i\in I}(a\dotdiv b_i)$. 
\end{Prop}
\begin{proof} By Proposition~\ref{adjunci\'on} (1), $a\dotdiv\bigwedge_{i\in I}b_i\leq \bigvee_{i\in I}(a\dotdiv b_i)$ implies that $a\leq\bigvee_{i\in I}(a\dotdiv b_i)+\bigwedge_{i\in I}b_i=\bigwedge_{i\in I}((\bigvee_{i\in I}(a\dotdiv b_i))+b_i)$, which holds by Proposition~\ref{adjunci\'on} (2) and Proposition~\ref{monotonia de la suma} (2) allow us to say that $a\leq (a\dotdiv b_j)+b_j\leq (\bigvee_{i\in I}(a\dotdiv b_i)) + b_j$ for any $j\in J$.\\
Since $j\in I$ was taken arbitrarily, then $a\leq\bigwedge_{i\in I}((\bigvee_{i\in I}(a\dotdiv b_i))+b_i)$.
\end{proof}

\begin{defi}
A $({\bf V},+,0)$ co-quantale is said to be a {\bf value co-quantale} if ${\bf V}$ is a value lattice. 
\end{defi}

\begin{defi}\label{Filtro de positivos}
Given ${\bf V}$ a value co-quantale, let set ${\bf V}^+=\{\epsilon\in {\bf V} : 0\prec \epsilon \}$, and we call it the {\bf positives filter} of ${\bf V}$.
\end{defi}

\begin{lem}\label{argumento epsilon medios}(\cite{Fla97}; Thrm 2.9])
If $({\bf V},+,0)$ is a value co-quantale, given $\epsilon\in{\bf V}^+$ there exists $\delta\in {\bf V}^+$ such that $\delta + \delta\prec \epsilon$
\end{lem}

By an obvious inductive argument, we can prove the following fact.

\begin{Coro}[\cite{LiRoZa18}; Remark~2.26]\label{argumentos epsilon sobre n}
Given any $\epsilon\in{\bf V}^+$ and $n\in\mathbb{N}\setminus \{0\}$, there exists $\theta\in{\bf V}^+$ such that $\stackbin{n \text{ times}}{n\theta:=\overbrace{\theta+\cdots+\theta}\prec\epsilon}$.
\end{Coro}

\begin{hec}(\cite{Fla97}; Thrm 2.10])\label{descender con la suma}
Given a value co-quantale ${\bf V}$ and $p\in {\bf V}$, $p=p+0=p+\bigwedge\{\epsilon\in V : 0\prec\epsilon\}=\bigwedge\{p+\epsilon : 0\prec\epsilon\}$.
\end{hec}

\begin{lem}(\cite{Fla97}; Thrm 2.11])
If ${\bf V}$ is a value co-quantale and $p,q\in{\bf V}$ such that $q\prec p$, then there exist $r,\epsilon\in{\bf V}$ such that $0\prec\epsilon$, $q\prec r$ and $r+\epsilon\prec p$
\end{lem}

\begin{Ej}
The 2-valued Boolen algebra ${\bf 2}=\{0,1\}$ is a value co-quantale, by taking $+:=\vee$
\end{Ej}

\begin{Ej}
$([0,\infty],\le,+)$ is a value co-quantale. Notice that the operation $-\dotdiv a$  ($a\in[0,\infty]$) is given by $b\dotdiv a:=max\{0,b-a\}$. We denote this example by $\mathcal{D}$.
\end{Ej}

\begin{Ej}
There are several ways to define a value co-quantale with underlying value lattice $\mathcal{E}$. One of them  is by taking $+$ as the usual real product, another one consists by taking $+:=\vee$.  Also, the {\L ukasiewicz's addition $+_L$} provides a value co-quantale structure with underlying lattice $\mathcal{E}$ -where $a+_Lb=\wedge\{0, a+b-1\}$ by taking $\wedge$ in $([0,1],\ge)$-. Denote the previous examples by $\mathcal{E}_*$, $\mathcal{E}_{\vee}$ and $\mathcal{E}_L$ respectively.
\end{Ej}

\begin{hec}(\cite{FlaKop97}; pg 115-117)
Given a non empty set $R$, $(\Omega(R),\supseteq,\cap)$ is a value co-quantale. In particular, if $(X,\tau)$ is a topological space, $(\Omega(\tau),\supseteq,\cap)$ is a value co-quantale. This last example is called the {\bf free local} associate to $(X,\tau)$.
\end{hec}

\begin{defi}\label{Cuantal Co-divisible}
A co-quantale $({\bf V},\leq,+)$ is said to be {\bf co-divisible} if for all $a,b\in {\bf V}$, $a\leq b$ implies that there exists $c\in {\bf V}$ such that $b=a+c$. 
\end{defi}

\begin{lem}\label{caracterizaci\'on de co-divisibilidad}
A co-quantale $({\bf V},\leq,+)$ is co-divisible,if and only, if for all $a,b\in {\bf V}$, $a\leq b$ implies $b=a+(b\dotdiv a)$.
\end{lem}
\begin{proof}
 This follow from $-\dotdiv a\dashv a+-$.
\end{proof}

The following property allows us to approximate $0$ by means of a $\mathbb{N}$-indexed sequence. 

\begin{defi}(\cite{LiRoZa18})\label{propiedad SAFA}
Given a co-quantale ${\bf V}$, we say that it has the {\bf Sequential Approximation From Above property} (shortly, $SAFA$), if and only, if there is a sequence $(u_n)_{n\in\mathbb{N}}$ such that
\begin{enumerate}
    \item $\bigwedge_{n\in\mathbb{N}}u_n=0.$
    \item for all $n\in\mathbb{N}$, $0\prec u_n$
    \item for all $n\in\mathbb{N}$, $u_{n+1}\leq u_n$
\end{enumerate}
\end{defi}

\subsubsection{Co-Girard value co-quantales.}
In this section, we give the basic notions and results relative to co-Girard value co-quantales, which allows to consider a kind of pseudo complement relative to some fixed element $b$ (a dualizing element) and $\dotdiv$. This notion will allow to prove our test questions (Proposition~\ref{test de Tarski-Vaught} and Theorem~\ref{Teorema de Los}).

\begin{defi}
Given a co-quantale $({\bf V}, +)$, an element $d\in V$ is said to be a {\bf dualizing element}, if and only if, for all $a\in V$ we have that  $a=d\dotdiv(d\dotdiv a)$.
\end{defi}

\begin{defi}\label{cuantal co-Girard}\index{Cuantal!co-Girard}
A co-quantale $ {\bf V}$ is said to be {\bf co-Girard}, if and only, if it has a dualizing element.
\end{defi}


\subsection{Continuity spaces}
In this section, we will provide some basics on {\it continuity spaces}, a framework given in \cite{Fla97,FlaKop97} in order to generalize (pseudo) metric spaces but considering distances with values in general co-quantales. 

\begin{defi}\label{espacios de continuidad}
Given $X\neq \emptyset$ a set, ${\bf V}$ a value co-quantale and a mapping $d:X\times X\rightarrow{\bf V}$, the pair $(X,d)$ is said to be a ${\bf V}$-{\bf continuity space}, if and only if,
\begin{enumerate}
    \item (reflexivity) for all $x\in X$, $d(x,x)=0$, and
    \item (transitivity) for any $x,y,z\in X$, $d(x,y)\leq d(x,z)+d(z,y)$.
\end{enumerate}

If there is no confusion about which co-quantale ${\bf V}$ we are considering, we may say that $(X,d)$ is just a {\bf continuity space}.
\end{defi}

\begin{Ej}\label{pre-orden como espacio}
Let ${\bf V}:=(2,,\le,\vee)$ and $X\neq \emptyset$ be a set. A $2$-continuity space $(X,d)$ codifies a binary relation $R:=\{(x,y) \in X\times X : d(x,y)=0\}$ on $X$ which is reflexive and transitive.
\end{Ej}

\begin{Ej}
Given $\mathcal{D}:=([0,\infty],\le,+)$, a $\mathcal{D}$-continuity space is just a  pseudo metric space (where the distance betwwen two elements might be infinite). 
\end{Ej}


\begin{Ej}\label{ExampleTop}
Given $(X,\tau)$ a topological space and $a,b\in X$, define  $d(a,b):=\{A\subseteq_{finite}\tau : $ for all $ U\in A,\ a\in U $ implies $ b\in U\}$. $(X,d)$ is a $\Omega(\tau)$-continuity space. 
\end{Ej}

\begin{Ej}
Given a value co-quantale ${\bf V}$, define  $d:{\bf V}\times{\bf V}\rightarrow{\bf V}$ by $d(a,b):=b\dotdiv a$. $(V,d)$ is a ${\bf V}$-continuity space.
\end{Ej}
%

\begin{Prop}\label{estructura en el producto} (\cite{FlaKop97}; Pg 119])
Given a value co-quantale {\bf V} and $(X,d_X), (Y,d_Y)$ {\bf V}-continuity spaces, define
\begin{center}
 $d_{X\times Y}: (X\times Y)\times(X\times Y)\rightarrow {\bf V}$  
\end{center} 
by $d_{X\times Y}((x_1,y_1),(x_2,y_2)):=d_X(x_1,x_2)\vee d_Y(y_1,y_2)$, then $(X\times Y, d_{X\times Y})$ is a ${\bf V}$-continuity space.
\end{Prop}

\begin{remark}\label{la distancia al 0 en el caso sim\'etrico}
There is another way to provide to a value co-quantale a ${\bf V}$-continuity space structure by defining $d^s_V(a,b):=(a\dotdiv b) \vee (b\dotdiv a)$ for all $a,b\in {\bf V}$. $d^s_V$ is called the {\bf symmetric distance} for $\mathbf{V}$. Notice that for all $a\in{\bf V}$ we have that $d^s_V(a,0)=d^s_V(0,a)=a$:
\begin{eqnarray*}
d^s_V(a,0)&=&(a\dotdiv 0)\vee (0\dotdiv a)\\
&=&(a\dotdiv 0) \vee 0 \ (\text{by Proposition~\ref{adjunci\'on} (4) and } \ 0=\min {\bf V}) \\
 &=& a\dotdiv 0\\
 &=& \bigwedge\{r: r+0 \geq a\} \ (\text{by definition of } \dotdiv )\\
 &=& \bigwedge\{r: r \geq a\} \\
 &=& a 
 \end{eqnarray*}
\end{remark}

\subsubsection{The topology of a $\mathbf{V}$-continuity space.}
In this subsection, we will give some basics on the underlying topology of a $\mathbf{V}$-continuity space.
\begin{nota}
Throughout this subsection, ${\bf V}$ will denote a value co-quantale.
\end{nota}

\begin{defi}\label{Bolas del espacio}
Given a ${\bf V}$-continuity space $(X,d)$, $\epsilon\in{\bf V}^+$ and $x\in X$ define $B_{\epsilon}(x):=\{y\in X : d(x,y)\prec\epsilon\}$ (which we call the {\bf disc} with radius $\epsilon$ centered in $x$.
\end{defi}

\begin{defi}
A subset $U$ of a $\text{V}$-continuity space $(X,d)$ is said to be {\bf open}, if and only if, given $x\in U$ there exists some $\epsilon\in{\bf V}^+$ such that $B_{\epsilon}(x)\subseteq U$.
\end{defi}

\begin{hec}(\cite{Fla97}; Thrm 4.2)
The family of open subsets in a ${\bf V}$-continuity space $(X,d)$ is closed under finite interesections and arbitrary unions. Also, $\emptyset$ and $X$ are open sets.
\end{hec}

\begin{defi}\label{topolog\'ia inducida}
Given $(X,d)$ a ${\bf V}$-continuity space, the family of open sets of $(X,d)$ determines a topology on $X$, which we will call the {\bf topology induced} by $d$ and we will denote it by $\tau_d$.
\end{defi} 

\begin{lem}
Given $(X,d)$ a ${\bf V}$-continuity space, for all $x\in X$ and $\epsilon\in{\bf V}^+$ we have that the disc $B_{\epsilon}(x)$ is an open set of $X$.
\end{lem}
\begin{proof}
Let $y\in B_{\epsilon}(x)$, so $d(x,y)\prec\epsilon$. By Fact~\ref{descender con la suma} $d(x,y)=\bigwedge\{d(x,y)+\delta : 0\prec \delta\}\prec\epsilon$, then by Lemma~\ref{continuidad} there exists $\delta\in{\bf V}^+$ such that $d(x,y)+\delta\prec\epsilon$. We may assure that $B_{\delta}(y)\subseteq B_{\epsilon}(x)$: If $z\in X$ satisfies $d(y,z)\prec\delta$, then $d(x,z)\leq d(x,y)+d(y,z)\leq d(x,y)+\delta\prec\epsilon$.
\end{proof}

\begin{hec}\label{base de la topolog\'ia}
Given a ${\bf V}$-continuity space $(X,d)$, the family of open discs forms a base for $\tau_d$.
\end{hec}

\begin{defi}\label{distancia entre un elemento y un conjunto}
Given a ${\bf V}$-continuity space $(X,d)$, $A\subseteq X$ and $x\in X$, define $d(x,A):=\bigwedge\{d(x,a) : a\in A\}$.
\end{defi}

\begin{Prop}\label{caracterizaci\'on de los cerrados}
Given a ${\bf V}$-continuity space $(X,d)$, then a subset $A\subseteq X$  is $\tau_d$-closed, if and only if, for all $x\in X$ we have that $d(x,A)=0$ implies $x\in A$.
\end{Prop}
\begin{proof}
Suppose that $A\subseteq X$ is $\tau_d$-closed and let $x\in X$ be such that  $d(x,A)=0$. In case that $x\notin A$, since $A$ is $\tau_d$-closed, there exists  $\epsilon\in{\bf V}^+$ such that $B_{\epsilon}(x)\subseteq A^c=\{y\in X : y\notin A \}$. Since $d(x,A):=\bigwedge\{d(x,a) : a\in A\}=0\prec\epsilon$, by Fact~\ref{continuidad} there exists $a\in A$ such that $d(x,a)\prec \varepsilon$, so $a\in B_{\epsilon}(x)\cap A$ (contradiction).\\
On the other hand, let $A\subseteq X$ be such that for all $x\in X$, $d(x,A)=0$ implies $x\in A$. Suppose that $y\in X$ belongs to the adherence of $A$, so by Fact~\ref{base de la topolog\'ia} given any $\epsilon\in{\bf V}^+$, we have that $A\cap B_{\epsilon}(y)\neq\emptyset$. Therefore, for any $\varepsilon\in V^+$ we have that  $d(y,A):=\bigwedge\{d(y,a) : a\in A\}\leq\varepsilon$, so $d(y,A)\le \bigwedge\{\varepsilon: 0\prec \varepsilon\}=0$, hence by hypothesis we may say that $y\in A$. Therefore, $A$ is closed.
\end{proof}

\begin{Coro}
Given a ${\bf V}$-continuity space $(X,d)$, the topological closure of $A\subseteq X$ is given by $cl(A):=\overline{A}=\{y\in X : d(y,A)=0\}$.
\end{Coro}

\begin{defi} 
Given  a ${\bf V}$-continuity space $(X,d_X)$, $\epsilon\in{\bf V}^+$ and $x\in X$, define the {\bf closed disc} of radius $\varepsilon$ centered in $x$ by $C_\epsilon(x):=\{y\in X : d_X(x,y)\leq\epsilon\}$.
\end{defi}

\begin{hec}\label{SisFundeVecindades}(\cite{FlaKop97}; Lemma 3.2 (2))
Let $(X,d_X)$ be a ${\bf V}$-continuity space and $x\in X$. The family $\{ C_\epsilon(x): \epsilon\in{\bf V}^+\}$ determines a  fundamental system of neighborhoods around $x$.
\end{hec}

\begin{defi}\label{distancia dual}
Given a ${\bf V}$-continuity space $(X,d)$, define the {\bf dual distance} $d^{\star}:X\times X\rightarrow{\bf V}$ relative to $d$ by $d^\star(x,y):=d(y,x)$. In general,  if we add $\star$ as a superscript to any topological notion, it means that it is related to the distance $d^{\star}$; e.g., the topology induced by $d^\star$ is denoted by $\tau_d^\star$.
\end{defi}

\begin{Prop}\label{Duales son cerrados}
Given a ${\bf V}$-continuity space $(X,d)$, $x\in X$  and $\epsilon\in{\bf V}^+$, $C^{\star}_{\epsilon}(x)$ is $\tau_d$-closed.
\end{Prop}
\begin{proof}
Let $x\in X$ and $\epsilon\in{\bf V}^+$, so by Proposition~\ref{caracterizaci\'on de los cerrados} it is enough to check that for any $y\in X$,  $d(y,C^{\star}_{\epsilon}(x)):=\bigwedge\{d(y,a): a\in C^{\star}_{\epsilon}(x)\}=0$ implies that $y\in C^{\star}_{\epsilon}(x)$. Let $y\in X$ be such that $d(y,C^{\star}_{\epsilon}(x))=0$ and $\delta\in{\bf V}^+$. Since $d(y,C^{\star}_{\epsilon}(x))=0\prec \delta$, by Fact~\ref{continuidad} there exists $z\in C^{\star}_{\epsilon}(x)$ such that $d(y,z)\prec \delta$, so $d^{\star}(x,y):=d(y,x)\leq d(y,z)+d(z,x)=d(y,z)+d^{\star}(x,z)\leq \delta +\epsilon$, therefore $d^{\star}(x,y)\leq\bigwedge\{\epsilon+\delta : 0\prec\delta\}=\epsilon+\bigwedge\{\delta : 0\prec\delta\}=\epsilon + 0=\epsilon$, therefore $y\in C^{\star}_{\epsilon}(x)$.
\end{proof}

\begin{defi}\label{distancia sim\'etrica}
Given a ${\bf V}$-continuity space  $(X,d)$, define the {\bf symmetric space} relative to $(X,d)$ by $(X,d^s)$, where $d^s(x,y):=d(x,y)\vee d^\star(x,y)$. In general, we will denote the topological notions related to $d^s$ by adding the superscript $s$.
\end{defi}

\begin{Prop}(\cite{FlaKop97}; Lemma 3)
Given a ${\bf V}$-continuity space $(X,d)$, for the topology $\tau^s$ induced by $d^s$ we have that  $U\subseteq X$ belongs to $\tau^s$, if and only if, there exist $V,W\subseteq X$ such that $V\in\tau_d$, $W\in\tau^{\star}$ and $U=V\cap W$.
\end{Prop}
%

\begin{lem}\label{Propiedades de separaci\'on}(\cite{FlaKop97}; Lemma 3,4)
If $(X,d)$ is a ${\bf V}$-continuity space, $(X,\tau^s)$ satisfies the following separation properties:
\begin{enumerate}
    \item (pseudo-Hausdorff) For all $x,y\in X$, if $x\notin \overline{\{y\}}$ according to $(X,\tau_d)$ then there exist $U,V\subseteq X$ such that $x\in U, y\in V, U\in \tau_d$, $V\in\tau^\star$ and $U\cap V=\emptyset$.
    \item (regularity) For all $x\in X$ and $A\subseteq X$, if $A\in\tau_d$ and $x\in A$ then there exist $U,C\subseteq X$ such that $U$ is $\tau_d$-open, $C$ es $\tau^\star$-closed and $x\in U\subseteq C\subseteq A$.
\end{enumerate}
\end{lem}

\subsubsection{{\bf V}-domains}
In order to give a version of \L o\'s Theorem in our setting, following~\cite{ChKe66,BeBeHeUs}, we need to consider compact and Hausdorff topological spaces. The setting which involves these assumptions in continuity spaces corresponds to {\it $\mathbf{V}$-domains}.

\begin{defi}
A ${\bf V}$-continuity space $(X,d)$ is said to be $T_0$, if and only if, for any $x,y\in X$,  $d(x,y)=0$ and $d(y,x)=0$ implies $x=y$.
\end{defi}

\begin{remark}[\cite{FlaKop97}, pg 120]\label{EquivT0}
A continuity space $(X,d)$ is $T_0$, if and only if, $(X,\tau_d)$ is $T_0$ as a topological space and $(X,\tau_d^s)$ is Hausdorff. 
\end{remark}

\begin{defi}\label{V-domain}
A ${\bf V}$-continuity space $(X,d)$ is said to be a ${\bf V-domain}$, if and only if,  it is $T_0$ and $(X,\tau_d^s)$ is compact.
\end{defi}

\begin{remark}\label{Domain-Hausdorff-Compact}
Let $(X,d)$ be a $\mathbf{V}$-domain, therefore by definition $(X,\tau_d^s)$ is compact. Since $(X,d)$ is $T_0$, by Remark~\ref{EquivT0} $(X,\tau_d^s)$ is Hausdorff.
\end{remark}
 
The importance of the previous properties lies on the fact that these allow us to provide a proof of a version of \L os's Theorem in the logic that we will introduce in this paper (Theorem~\ref{Teorema de Los}). We will provide some examples which satisfy these properties.

\begin{Prop}(\cite{FlaKop97}; Thrm 4.14)
The following examples are domains:
\begin{enumerate}
    \item $2=(\{0,1\}, 0\leq 1, \vee)$.
    \item $([0,1],\leq, +)$.
    \item The quantale of errors $([0,1],\geq, \otimes)$, where $a\otimes b:=max\{a+b-1,0\}$.
    \item The quantale of  fuzzy subsets associated to a set $X$\footnote{It is denoted by $\Lambda(X)$in~\cite{FlaKop97}}.
    \item The free local associated to a set $X$: $(\Omega(X),\supseteq, \cap)$\footnote{Tt is denoted by $\Gamma(X)$ in~\cite{FlaKop97}}.
\end{enumerate}
 \end{Prop}
 
\section{Value co-quantale logics}\label{VLogics}

In this section, we will introduce a logic with truth values within value co-quantales, generalizing Continuous Logic (see~\cite{BeBeHeUs}, where the truth values are taken in the unitary interval $[0,1]$, which is a particular case of our setting).
\ \\ \ \\
Throughout the rest of this paper, we assume some technical conditions (Definitions~\ref{Cuantal Co-divisible}, \ref{cuantal co-Girard} and \ref{V-domain}) that we need for providing a proof of a version of Tarski-Vaught test -Proposition~\ref{test de Tarski-Vaught}- and a version of \L o\'s Theorem -Theorem~\ref{Teorema de Los}- for the logics introduced in this paper. At some point, we require to work with the symmetric distance $d^s_V$ of $\mathbf{V}$. 

\begin{assumption}
Throughout this section, we assume that ${\bf V}$ is a value co-quantale which is co-divisible, co-Girard and a $\mathbf{V}$-domain.
\end{assumption}

\subsection{Modulus of uniform continuity}
Modulus of uniform continuity are introduced in~\cite{BeBeHeUs} as a technical way of controlling from the language the uniform continuity of the mappings considered in Continuous Logic. In this subsection, we develop an analogous study of modulus of uniform continuity but in the setting of mappings valued in value co-quantales.

\begin{remark}
Given $(M,d_M)$, $(N,d_N)$ ${\bf V}$-continuity spaces and\linebreak
$(x_1,y_1),(x_2,y_2)\in M\times N$, we define $d_{M\times N}((x_1,y_1), (x_2,y_2)):=d_M(x_1,y_1)\vee d_N(x_2,y_2)$. By Proposition~\ref{estructura en el producto}, $(M\times N,d_{M\times N})$ is a ${\bf V}$-continuity space. 
\end{remark}
%

\begin{defi}(c.f. \cite{BeBeHeUs}; pg 8)\label{modunif}\index{Modulus of uniform continuity}
 Given a mapping $f:M\rightarrow N$ between two ${\bf V}$-continuity spaces $(M,d_M)$ and $(N,d_N)$, we say that $\Delta:{\bf V}^+\rightarrow{\bf V}^+$ is a {\bf modulus of uniform continuity} for $f$, if and only if, for any $x,y \in M$ and any $\epsilon\in{\bf V}^+$, $d_M(x,y)\leq\Delta(\epsilon)$ implies $d_N(f(x),f(y))\leq\epsilon$.
\end{defi}

As a basic consequence we have the following fact.

\begin{Prop}\label{CompositionModulus}
Given  $(M,d_M), (N,d_N)$, $(K,d_K)$ ${\bf V}$-continuity spaces and  $f:M\rightarrow N$, $g:N\rightarrow K$ uniformly continuous mappings, $\Delta$ and $\Theta$ modulus of uniform continuity for $f$ and $g$ respectively, then $\Delta\circ\Theta$ is a modulus of uniform continuity for $g\circ f$.
\end{Prop}
%

\begin{defi}
 Given a sequence of mappings $(f_n)_{n\in \mathbb{N}}$ with domain $(M,d_M)$ and codomain $(N,d_N)$ (both of them ${\bf V}$-continuity spaces), we say that $(f_n)_{n\in \mathbb{N}}$ {\bf uniformly converges} to a mapping $f:M\rightarrow N$, if and only if, for all $\epsilon\in{\bf V}^+$ there exists $n\in\mathbb{N}$ such that for any $m\geq n$ and for any $x\in M$ we may say that $d_N(f_m(x),f(x))\leq\epsilon$.
\end{defi}


It is straightforward to see that uniform convergence behaves well with respect to composition of mappings. 

\begin{Prop}
Let $(M,d_M),(N,d_N)$, $(K,d_K)$ be ${\bf V}$-continuity spaces, $f:M\to N$,  $(f_n)_{n\in \mathbf{N}}$ be a sequence of mappings from $M$ to $N$, $g:N\to K$, and $(g_n)_{n\in \mathbb{N}}$ be a sequence of mappings from $N$ to $K$ such that $(f_n)_{n\in \mathbb{N}}$ uniformly converges to $f$ and $(g_n)_{n\in \mathbb{N}}$ uniformly converges to $g$. If $g$ is uniformly continuous, then $(g_n\circ f_n)_{n\in \mathbb{N}}$ uniformly converges to $g\circ f$.
\end{Prop}

\subsubsection{Uniform continuity of $\bigwedge$ and $\bigvee$.}

The following fact is very important because, as in Continuous Logic, it allows us to control (by using directly the language) the uniform continuity of both $\bigwedge$ ($\inf$) and $\bigvee$ ($\sup$), understood as quantifiers (in an analogous way as in Continuous Logic).

\begin{Prop}\label{caso de los cuantificadores}
Let $(M,d_M)$, $(N,d_N)$ be ${\bf V}$-continuity spaces, $f:M\times N \rightarrow {\bf V}$ be a uniformly continuous mapping provided with a modulus of uniform continuity $\Delta:{\bf V}^+ \rightarrow {\bf V}^+$, then $\Delta$ is also a modulus of uniform continuity for the mappings $\bigvee_f: M \rightarrow {\bf V}$ and $\bigwedge_f: M \rightarrow {\bf V}$ defined by $x \mapsto \bigvee_{y \in N} f(x,y)$ and $x \mapsto \bigwedge_{y \in N} f(x,y)$ respectively.
\end{Prop}
%
 %
\begin{proof}
Let $\epsilon \in {\bf V}$ be such that $0\prec \epsilon$, $y\in N$ and $a,b \in M$ be such that $d_{M}(b,a)\leq \Delta(\epsilon)$. Then, 
\begin{eqnarray*}
d_{M\times N}((b,y), (a,y))&:=&d_{M}(b,a)\vee d_N(y,y)=d_{M}(b,a)\\
&\le&\Delta(\epsilon)\\
\end{eqnarray*}

Since $\Delta$ is a modulus of uniform continuity for $f$, then

\begin{eqnarray*}
f(a,y)\dotdiv f(b,y)&\leq& d_{V}(f(b,y),f(a,y))\\
&\leq&\epsilon
\end{eqnarray*}

By Proposition~\ref{adjunci\'on} (1) we may say 

$$f(a,y)\leq f(b,y) + \epsilon\leq \bigvee_{z\in N}f(b,z) + \epsilon$$

Since $y\in N$ was taken arbitrarily, then

$$\bigvee_{z\in N}f(a,z)\leq \bigvee_{z\in N}f(b,z) + \epsilon$$

and by Proposition~\ref{adjunci\'on} (1) 

$$\bigvee_{f} f(a)\dotdiv\bigvee_{f} f(b)=\bigvee_{z\in N}f(a,z)\dotdiv\bigvee_{z\in N}f(b,z)\leq\epsilon.$$

%
In a similar way we prove the related statement for $\bigwedge_f$.
%
\end{proof}

As an immediate consequence, we have the following useful facts.
%
\begin{Coro}\label{corolario importante}
Given an arbitrarily set $I\neq \emptyset$ and $I$-sequences $(a_i)_{i\in I}, (b_i)_{i\in I}$ in a $\text{V}$-continuity space $(M,d_M)$, if $\epsilon\in{\bf V}^+$ satisfies $d_V(a_i,b_i)\leq \epsilon$ for all $i\in I$, then $d_V(\bigvee_{i\in I}a_i,\bigvee_{i\in I}b_i)\leq \epsilon$ and $d_V(\bigwedge_{i\in I}a_i,\bigwedge_{i\in I}b_i)\leq \epsilon$.
\end{Coro} 


\begin{Coro}
Given a ${\bf V}$-continuity space $(M,d_M)$, $I\neq \emptyset$, and a $I$-sequence of mappings $(f_i:M\rightarrow {\bf V})_{i\in I}$, if $\Delta:{\bf V}^+\rightarrow{\bf V}^+$ is a modulus of uniform continuity for $f_i$ ( $i\in I$), then $\Delta$ is also a modulus of uniform continuity for both $\bigvee_i f_i:M\rightarrow{\bf V}$ and $\bigwedge_i f_i:M\rightarrow{\bf V}$ defined by $x\mapsto \bigvee_{i\in I} f_i(x)$ and $x\mapsto\bigwedge_{i\in I} f_i(x)$, respectively.
\end{Coro}

\begin{Prop}
Let $(M,d_M)$, $(N,d_N)$ be ${\bf V}-$continuity spaces, $f:M\times N\to {\bf V}$, $(f_n)_{n \in \mathbb{N}}$ a sequence of mappings from $M\times N$ to ${\bf V}$ such that $(f_n )_{n \in \mathbb{N}}$ uniformly converges to $f$, then $\left( \bigvee_{y \in N} f_n(x,y)\right)_{n \in \mathbb{N}}$ uniformly converges to $\bigvee_{y \in N}f(x,y)$ and $\left( \bigwedge_{y \in N} f_n(x,y)\right)_{n \in \mathbb{N}}$ uniformly converges to $\bigwedge_{y \in N}f(x,y)$.
\end{Prop}
\begin{proof}
Since by hypothesis $(f_n)_{n\in \mathbb{N}}$ uniformly converges to $f$, given $\epsilon\in{\bf V}^+$ there exists $n\in\mathbb{N}$ such that if $m\geq n$, then $d_V(f_m(x,y),f(x,y))\leq\epsilon$ for any $(x,y)\in M\times N$. For a fixed $x\in M$ and $m\geq n$, define the sequences $(f_m(x,y))_{y\in N}$ and $(f(x,y))_{y\in N}$, which satisfy the hypothesis of Corollary~\ref{corolario importante}, so $d_V(\bigvee_{y\in N}f_m(x,y),\bigvee_{y\in N}f(x,y))\leq\epsilon$ whenever $m\geq n$. Since this holds for all $\epsilon\in{\bf V}^+$, we got the uniform convergence desired.

In an analogous way, we prove the respective statement for $\bigwedge$.
\end{proof}

\subsection{Some basic notions.}
In first order logic, an $n$-ary relation in a set $A$ is defined as a subset of $A^n$. In this way, a tuple $(a_1,\cdots,a_n)$ might belong to $A$ or not. We may codify this by using characteristic functions, dually, by the discrete distance from a tuple in $A^n$ to $R$. In Continuous Logic, an $n$-ary relation in $A$ is understood according to this second approach by taking a uniformly continuous mapping $R:A^n\to [0,1]$. In this setting, we generalize this approach replacing $[0,1]$ by a suitable value co-quantale ${\bf V}$.

All topological notions about $\mathbf{V}$ are relative to the symmetric topology of $\mathbf{V}$. 

\begin{defi}\label{diam\'etro de un espacio}
Given a ${\bf V}$-continuity space $(M,d_M)$ and $A\subseteq M$, define $diam(A):=\bigvee\{d_M(a,b) : a,b\in A\}$.  (which we will call the {\bf diameter} of $A$.
\end{defi}

\subsubsection{Continuous structures.}
Given a ${\bf V}$-continuity space $(M,d_M)$ with diameter $p\in{\bf V}$, we define a \textit{continuous structure} with underline ${\bf V}$-continuity space $(M,d_M)$ as a tuple $\mathcal{M}=((M,d_M),(R_i)_{i\in I},(f_j)_{j\in J},(c_k)_{k\in K})$, where:
\begin{enumerate}
    \item For each $i\in I$, $R_i: M^{n_i}\rightarrow {\bf V}$ is a uniformly continuous mapping (which we call a {\bf predicate}), with modulus of uniform continuity $\Delta_{R_i}:{\bf V}^+\rightarrow{\bf V}^+$. In this case, $n_i<\omega$ is said to be the {\bf arity} of $R_i$.
    \item For each  $j\in J$, $f_j: M^{m_j}\rightarrow M$ is a uniformly continuous mapping with modulus of uniform continuity $\Delta_{F_j}:{\bf V}^+\rightarrow{\bf V}^+$. In this case, $m_j<\omega$ is said to be the {\bf arity} of $F_j$.
    \item For each $k\in K$, $C_k$ is an element $M$.
\end{enumerate}
%

\subsubsection{Languages for continuous structures.}
For a fixed  continuous structure $\mathcal{M}:=((M,d_M),(R_i)_{i\in I},(f_j)_{j\in J},(c_k)_{k\in K})$, we will define the {\bf language} associated to $\mathcal{M}$ in the natural way, as follows. 

\begin{center}
    Predicate symbols: $R_i\mapsto(P_i,n_{i},\Delta_{R_i})$ ($i\in I$)\\
    Function symbols: $f_j\mapsto(F_j,n_{j},\Delta_{f_j})$ $(j\in J)$\\
    Constant symbols: $c_k\mapsto e_k$ $(k\in K)$.
\end{center}

This set of {\bf non logical symbols} is denoted by $\mathbf{NL}_{\mathcal{M}}$.

Let us denote by $\mathbf{LG}:=\{d\}\cup X \cup C\cup\{\bigvee,\bigwedge\}$ 
(which we call {\bf logical symbols}), where:

\begin{itemize}
    \item $X=\{x_i : i\in \mathbb{N}\}$ is a countable set of {\bf variables}.
    \item $C$ is the set of all uniformly continuous mappings with domain $\mathbf{V}^n$ and codomain $\mathbf{V}$ ($1\le n<\omega$). As in Continuous Logic, we understand a uniformly continuous mapping $u:\mathbf{V}^n\to \mathbf{V}$ as a {\bf connective}. 
    \item $d$ is a symbol, which we will interpret as the $\mathbf{V}$-valued distance given in $(M,d_M)$. This symbol will play the role of the equality in first order logic, in a similar way as we do in Continuous Logic. 
\end{itemize}

\begin{remark}
For the sake of simplicity, as in \cite{BeBeHeUs}, it would be nice to know if there is a uniformly dense subset of complete connectives in this setting. In this paper, we do not focus on this question. 
\end{remark}

\begin{defi}\label{lenguaje para estructuras V-valuadas}
 Given a continuous structure $\mathcal{M}:=$\linebreak
 $((M,d_M),(R_i)_{i\in I},(f_j)_{j\in J},(c_k)_{k\in K})$, we define the {\bf language} based on $\mathcal{M}$ as $\mathbf{L}_{\mathcal{M}}:={\bf NL}_{\mathcal{M}}\cup{\bf LG}$. We will drop $\mathcal{M}$ if it is clear from the context. 
\end{defi}

We define the notion of terms as follows.

\begin{defi}\label{L-t\'ermino}
Given a language based on a continuous structure ${\bf L}$, we define the notion of {\bf L}-term recursively, as follows:
\begin{itemize}
     \item Any variable and any constant symbol is an $\mathbf{L}$-term.
     \item Given  ${\bf L}$-terms $t_1,...,t_n$ and a function symbol $f\in \mathbf{L}$ of arity $n$, $ft_1,...,t_n$ is an ${\bf L}$-term.
\end{itemize}
\end{defi}

\begin{defi}
 An $\mathbf{L}$-term is said to be {\bf closed}, if and only if, it is built without use of variables.
\end{defi}

Now, we provide the notion of $\mathbf{L}$-formulae in this new setting. We mimic the analogous notion given in Continuous Logic.

\begin{defi}\label{L-f\'ormula}
Given $\mathbf{L}$ a language based on a continuous structure, we define the notion of $\mathbf{L}$-formula recursively, as follows:
\begin{itemize}
    \item Given ${\bf L}$-terms $t_1,t_2$, $dt_1t_2$ is an $\mathbf{L}$-formula.
    \item Given ${\bf L}$-terms $t_1,...t_n$ and a predicate symbol $P\in \mathbf{L}$ of arity $n$, $Pt_1,...t_n$ is an $\mathbf{L}$-formula.
    \item Given ${\bf L}$-formulas $\psi_1,...,\psi_m$ and a connective (i.e., a uniformly continuous mapping) $a:\mathbf{V}^m\to \mathbf{V}$, then $a\psi_1,...,\psi_m$ is an $\mathbf{L}$-formula.
    \item Given an {\bf L}-formula $\psi$ and a variable $x$, both $\bigwedge x\psi$ and $\bigvee x\psi$ are $\mathbf{L}$-formulas.
\end{itemize}
\end{defi}

\begin{remark}\label{b-x mapping}
Let $\mathbf{V}$ be a co-Girard value co-quantale and $b\in {\bf V}$ be a dualizing element. Denote $x':=b\dotdiv x$. Denote the usual, dual and symmetric distances in $\mathbf{V}$ by $d$, $d^{*}$ and $d^s$ respectively. Notice that the mapping $b\dotdiv\square:\mathbf{V}\to \mathbf{V}$ defined by $(b\dotdiv\square)(x):=b\dotdiv x$ is uniformly continuous (relative to the symmetric topology) provided with modulus of uniform continuity $id_{\mathbf{V}^+}$.
In fact, given $x,y\in \mathbf{V}$ we have that
\begin{eqnarray*}
d(y,x) &=& x\dotdiv y\\
&=& (b\dotdiv (b\dotdiv  x))\dotdiv y) \text{ ($b$ is a dualizing element)}\\
&=& b\dotdiv((b\dotdiv  x)+ y) \text{ (by Prop.~\ref{adjunci\'on} (5))}\\
&=& b\dotdiv(y+(b\dotdiv  x)) \text{ ($+$ is commutative)}\\
&=& (b\dotdiv y)\dotdiv (b\dotdiv x)  \text{ (by Prop.~\ref{adjunci\'on} (5))}\\
&=& y'\dotdiv x'\\
&=& d(x',y')
\end{eqnarray*}
Therefore, $d(y,x)=d(x',y')=d^{*}(y',x')$.

Exchanging the role of $y$ and $x$ above, we may say that $d(x,y)=d^{*}(y,x)=d^{*}(x',y')=d(y',x')$. Since $d^s(x,y)=d(y,x)\vee d^{*}(y,x)=d^{*}(y',x')\vee d(y',x')=d^s(y',x')$, then $(b\dotdiv \square)$ respects the distance $d^s$ and therefore it is uniformly continuous relative to the symmetric topology provided with $id_{\mathbf{V}^+}$ as a modulus of uniform continuity. 
\end{remark}

\begin{notation}\label{notacion para existe y para todo}
The subsequences $\bigvee x$ and $\bigwedge x$ of an ${\bf L}$-f\'ormula can be written as $\bigvee_x$ and $\bigwedge_x$, respectively.\\
$d(t_1,t_2)$ denotes the sequence $dt_1t_2$.
\end{notation}

\begin{defi}
An ${\bf L}$-formula $\phi$ is said to be {\bf quantifier-free}, if and only if, there are no appearances of $\bigvee_x$ and $\bigwedge_x$ inside $\phi$.
\end{defi}

\begin{defi}
An appearance of a variable $x$ inside an ${\bf L}$-f\'ormula $\phi$ is said to be {\bf free} whenever it is not under the scope of $\bigvee_x$ o $\bigwedge_x$ inside $\phi$. \end{defi}

\begin{defi}
An ${\bf L}$-formula $\phi$ is said to be an ${\bf L}$-{\bf sentence}, if and only if, all appearances of variables are not free.
\end{defi}

\begin{notation}
$\phi(x_1,...,x_n)$ means that the variables that appear free in $\phi$ are among $x_1,...,x_n$.
\end{notation}

\subsubsection{{\bf L}-structures} 
Let $\mathbf{L}$ be a language based on a $\mathbf{V}$-continuous structure $\mathcal{M}:=((M,d_M),(R_i)_{i\in I},(f_j)_{j\in J},(c_k)_{k\in K})$. Given a $\mathbf{V}$-continuity space $(N,d_N)$ with diameter (Definition~\ref{diam\'etro de un espacio}) at most $diam({M})$, we will interpret the symbols in $\mathbf{NL}_{\mathcal{M}}$ in $(N,d_N)$ as follows: 
\begin{itemize}
    \item For any predicate symbol $P$ of arity $n$ and modulus of uniform continuity $\Delta_P$,  associate a uniformly continuous mapping $P^{\mathcal{N}}:N^n\rightarrow{\bf V}$ with modulus of uniform continuity $\Delta_P$.
    \item For any function symbol $F$ of arity $m$ and with modulus of uniform continuity $\Delta_F$, associate a uniformly continuous mapping\\ $F^{\mathcal{N}}:N^m\rightarrow N$ with modulus of uniform continuity $\Delta_F$.
    \item For any constant symbol $e$, associate an element $e^{\mathcal{N}}\in N$.\\
\end{itemize}

Also, the logical symbol $d$ is interpreted in $(N,d_N)$ as the distance $d^\mathcal{N}:=d_N$.

\begin{defi}\label{L-estructura}
Given a language  ${\bf L}$ based on a continuous structure $\mathcal{M}:=((M,d_M),(R_i)_{i\in I},(f_j)_{j\in J},(c_k)_{k\in K})$ and a $\mathbf{V}$-continuity space  $(N,d_N)$, the continuous structure obtained by interpreting the symbols of $\mathbf{L}$ on $\mathcal{N}=((N,d_N)$  as above, $\mathcal{N}:=((N,d_N),(P^{\mathcal{N}}_i)_{i\in I},(F^{\mathcal{N}}_j)_{j\in J},(e_k)_{k\in K})$, is said to be an ${\bf L}$-{\bf structure}.
\end{defi}

\subsubsection{Semantics}
Given an {\bf L}-structure $\mathcal{N}$ and $A \subseteq N$, we extend the language ${\bf L}$ by adding new constant symbols $c_a$ ($a\in A$), interpreting  $c^{\mathcal{N}}_a:=a$. Abusing of notation, we will write $a$ instead of $c_a$, but understood as a constant symbol. Let us denote this language by ${\bf L}(A)$.

\begin{defi}
Given an ${\bf L}-$structure $\mathcal{N}$, for any $\mathbf{L}-term$ $t$ we define recursively its {\bf interpretation} in $\mathcal{N}$, denoted by $t^{\mathcal{N}}$, as follows:
\begin{enumerate}
    \item If $t$ is a constant symbol $c$, define $t^\mathcal{N}:=c^\mathcal{N}$.
    \item If $t$ is  a variable $x$, define $t^\mathcal{N}:N\rightarrow N$ as the identity function of $N$.
    \item If $t$ is of the form $ft_1,...t_n$ provided that $f$ is an $n$-ary function symbol and $t_1(\overline{x}),...,t_n(\overline{x})$ are {\bf L}-terms, define $t^\mathcal{N}:=f^\mathcal{N}(t_1^\mathcal{N},...,t_n^\mathcal{N})$ as the mapping $f^\mathcal{N}(t_1^\mathcal{N},...,t_n^\mathcal{N}):N^{m}\rightarrow N$ where $(\overline{a})\mapsto f^\mathcal{N}(t_1^\mathcal{N}(\overline{a}),...,t_n^\mathcal{N}(\overline{a}))$ for all $\overline{a}\in N^m$. 
\end{enumerate}
\end{defi}

\begin{defi}\label{interpretaci\'on de sentencias}
Let $\mathcal{N}$ be an {\bf L}-structure. We define recursively the {\bf interpretation} of ${\bf L}(N)$-sentences in $\mathcal{N}$, as follows.
\begin{enumerate}
    \item $(d(t_1,t_2))^\mathcal{N}:=d^{\mathcal{N}}(t^{\mathcal{N}}_1,t^{\mathcal{N}}_2)$, where  $t_1, t_2$ are ${\bf L}(\mathcal{N})$-terms
    \item $(P(t_1,...,t_n))^{\mathcal{N}}:=P^{\mathcal{N}}(t^{\mathcal{N}}_1,...,t^{\mathcal{N}}_n)$, where $P$ is an $n$-ary predicate symbol and $t_1,\cdots, t_n$ are ${\bf L}(\mathcal{N})$-terms.
    \item $(u(\phi_1,..,\phi_n))^{\mathcal{N}}:=u(\phi^{\mathcal{N}}_1,...,\phi^{\mathcal{N}}_n)$ for any uniformly continuous mapping (connective) $u:{\bf V}^n\rightarrow{\bf V}$ and all $\mathbf{L}({\mathcal{N}})$-sentences $\phi_1,...,\phi_n$.
    \item $(\bigvee_x\phi)^{\mathcal{N}}:=\bigvee_{a\in N}\phi^{\mathcal{N}}(a)$, whenever $\phi(x)$ is an ${\bf L}(\mathcal{N})$-formula.
    \item $(\bigwedge_x\phi)^{\mathcal{N}}:=\bigwedge_{a\in N}\phi^{\mathcal{N}}(a)$, whenever $\phi(x)$ is an ${\bf L}(\mathcal{N})$-formula.
\end{enumerate}
\end{defi}

Analogously as in Continuous Logic, all terms and all formulae have a modulus of uniform continuity, which do not depend of the structures.

\begin{Prop}
Given ${\bf L}$ a language based on a continuous structure, $\phi(x_1,...,x_n)$ an $\mathbf{L}-formula$ and $t(x_1,...,x_m)$ an $\mathbf{L}$-term, then there exist  $\Delta_{\phi}:{\bf V}^+\rightarrow {\bf V}^+$ and $\Delta_t:{\bf V}^+\rightarrow{\bf V}^+$ such that for any {\bf L}-structure $\mathcal{N}$, $\Delta_{\phi}$ is a modulus of uniform continuity for $\phi^{\mathcal{N}}$ and $\Delta_t$ is a modolus of continuity for $t^{\mathcal{N}}$.
\end{Prop}
\begin{proof}
The basic cases are given by definition, since constant symbols's interpretations can be viewed as constant functions, variables are interpreted as the identity function and predicate symbols are interpreted as a uniformly continuous mapping with the respective modulus of uniform continuity. The connective case corresponds to compose uniformly continuous mappings (and we get the desired result by Proposition~\ref{CompositionModulus}), and the quantifier cases follow from Proposition~\ref{caso de los cuantificadores}.
\end{proof}

\begin{defi}\label{sub-estructuras}
Given {\bf L} a language based on a continuous structure and  $\mathcal{M},\mathcal{N}$ $\mathbf{L}-structures$, we say that $\mathcal{M}$ is an $\mathbf{L}$-{\bf substructure} of $\mathcal{N}$, if and only if, $M\subseteq N$ and the interpretations of all non logical symbols and of $d$ in $\mathcal{M}$ correspond to the respective restrictions of the interpretations in $\mathcal{N}$ of those symbols.
\end{defi}

\subsubsection{$\mathbf{L}$-conditions.}

Fix $\mathbf{L}$ a language based on a continuous structure.

\begin{defi}(c.f. \cite{BeBeHeUs}; Def 3.9)
 Given $\phi_1(x_1,...,x_n),\phi_2(x_1,...,x_n)$ ${\bf L}-formulas$, we say that $\phi_1$ is {\bf logically equivalent} to $\phi_2$, if and only if, for any {\bf L}-structure $\mathcal{M}$ and any $a_1,...,a_n\in M$ we have that $\phi^\mathcal{M}_1(a_1,...,a_n)=\phi^\mathcal{M}_2(a_1,...,a_n)$.
\end{defi}

\begin{defi}\label{distancia l\'ogica entre dos f\'ormulas}
Let $\phi_1(x_1,...,x_n), \phi_2(x_1,...,x_n)$ be ${\bf L}$-formulas and $\mathcal{M}$ be an ${\bf L}$-structure, we define the {\bf logical distance} between $\phi_1$ and $\phi_2$ {\bf relative to $\mathcal{M}$} as follows:
\begin{center}
    $d(\phi_1,\phi_2)_\mathcal{M}:=\bigvee\{d_V(\phi^\mathcal{M}_1(a_1,...a_n),\phi^\mathcal{M}_2(a_1,...a_n)): a_1,...a_n\in M \}$
\end{center}
 The {\bf logical distance} between $\phi_1,\phi_2$ is defined as follows:
 \begin{center}
     $d(\phi_1,\phi_2):=\bigvee\{d(\phi_1,\phi_2)_\mathcal{M} : \mathcal{M}\ \text{is an} \ {\bf L}-\text{structure}\}$
 \end{center}
\end{defi}

We define the notion of satisfiability in an $\mathbf{L}$-structure in an analogous ways as in Continuous Logic, by using the notion of $L$-conditions.

\begin{defi}(c.f. \cite{BeBeHeUs}; pg 19)
An {\bf L-condition}  $E$ is a formal expression of the form $\phi=0$, where $\phi(x_1,...,x_n)$ is an {\bf L}-f\'ormula. An {\bf L}-condition $E$ is said to be {\bf closed} if it is of the form $\phi=0$, where $\phi$ is an {\bf L}-sentence. Given an $\mathbf{L}$-formula $\phi(x_1,...,x_n)$, the related condition $E:\phi(x_1,...,x_n)=0$ is denoted by $E(x_1,...,x_n)$.
\end{defi}

\begin{defi}\label{validez}
Given $\phi(x_1,...,x_n)$ an {\bf L}-f\'ormula, $\mathcal{M}$ an ${\bf L}$-structure and $a_1,..,a_n\in M$, the ${\bf L}$-condition $E(x_1,...,x_n): \phi(x_1,...,x_n)=0$ is said to be {\bf satisfied} in $\mathcal{M}$ for $a_1,...,a_n$, if and only if, $\phi^\mathcal{M}(a_1,...,a_n)=0$. We denote this by $\mathcal{M}\models E(a_1,...,a_n)$.
\end{defi}

\begin{notation}
Given $\phi,\psi$ $\mathbf{L}$-formulae, we denote by $\phi=\psi$ the {\bf L}-condition $(\phi\dotdiv\psi)\vee(\psi\dotdiv\phi)=0$ and we denote by $\phi\leq\psi$  the $\mathbf{L}$-condition $\phi\dotdiv\psi=0$.
\end{notation}

\begin{defi}
An {\bf L}-{\bf theory} is a set of closed {\bf L}-conditions.
\end{defi}

\begin{defi}
Given an {\bf L}-theory $T$ and an {\bf L}-structure $\mathcal{M}$, we say that $\mathcal{M}$ is a {\bf model} of $T$, if and only if, for any {\bf L}-condition $E\in T$ we have that $\mathcal{M}\models E$.
\end{defi}

\section{Tarski-Vaught test.}

\begin{defi}\label{equivalencia e inmersi\'on elemental}(c.f. \cite{BeBeHeUs}; Def 4.3)
Let $\mathcal{M},\mathcal{N}$ be {\bf L}-structures.
\begin{enumerate}
    \item We say that $\mathcal{M}$ is {\bf elementary equivalent} to $\mathcal{N}$ (denoted by $\mathcal{M}\equiv\mathcal{N}$), if and only if, any {\bf L}-sentence $\varphi$ satisfies $\varphi^\mathcal{M}=\varphi^\mathcal{N}$.
    \item Let $\mathcal{M}$ be an $\mathbf{L}$-substructure of $\mathcal{N}$. We say that $\mathcal{M}$ is an ${\bf L}$-{\bf elementary substructure} of $\mathcal{N}$ (denoted by $\mathcal{M}\preccurlyeq\mathcal{N}$), if and only if,  any ${\bf L}$-formula $\varphi(x_1,...,x_n)$ satisfies $\varphi^\mathcal{M}(a_1,...,a_n)=\varphi^\mathcal{N}(a_1,...,a_n)$ for all $a_1,...,a_n\in M$. In this case, we also say that $\mathcal{N}$ is an $\mathbf{L}$-elementary extension of $\mathcal{M}$.
\end{enumerate}
\end{defi}

We will provide a version of the well-known Tarski-Vaught test, as a equivalence of being an $\mathbf{L}$-elementary substructure, as it holds in both first order and Continuous logics. We need to assume that $\mathbf{V}$ is co-Girard (Definition~\ref{cuantal co-Girard}).

\begin{Prop}\label{test de Tarski-Vaught}(Tarski-Vaught test, c.f. \cite{BeBeHeUs} Prop 4.5)\; 
Assume that {\bf V} is a co-Girard value co-quantale and let $b$ a dualizing element of ${\bf V}$. Let $\mathcal{M}$, $\mathcal{N}$ be ${\bf L}$-structures such that $\mathcal{M}\subseteq \mathcal{N}$. The following are equivalent:
\begin{enumerate}
    \item $\mathcal{M}\preccurlyeq\mathcal{N}$.
    \item For any $\mathbf{L}$-formula $\varphi(x,x_1,...,x_n)$ and $a_1,...,a_n\in M$, we have that
    \begin{center}
        $\bigwedge\{\varphi^\mathcal{M}(c, a_1,...,a_n) : c\in M\}=\bigwedge\{\varphi^\mathcal{N}(c, a_1,...,a_n) : c\in N\}$
    \end{center}
\end{enumerate}
\end{Prop}
\begin{proof}
Suppose that $\mathcal{M}\preccurlyeq\mathcal{N}$. Let $\varphi(x,x_1,...,x_n)$ be an {\bf L}-formula and $a_1,...a_n\in M$, so
\begin{eqnarray*}
\bigwedge\{\varphi^\mathcal{M}(c, a_1,...,a_n) : c\in M\}&=& (\bigwedge_x\varphi(x,a_1,...,a_n))^\mathcal{M}\\
&& (\text{by Definition~\ref{interpretaci\'on de sentencias}} (5)) \\
&=& (\bigwedge_x\varphi(x,a_1,...,a_n))^\mathcal{N} \\
&&(\text{since } \mathcal{M}\preccurlyeq\mathcal{N})\\
&=& \bigwedge\{\varphi^\mathcal{N}(c, a_1,...,a_n) : c\in N\} \\
&& (\text{by Definition~\ref{interpretaci\'on de sentencias}} (5)). 
\end{eqnarray*}
On the other hand, suppose that for any ${\bf L}$-formula $\varphi(x,x_1,...,x_n)$ and any $a_1,...,a_n\in M$ we have that
        $$\bigwedge\{\varphi^\mathcal{M}(c, a_1,...,a_n) : c\in M\}=\bigwedge\{\varphi^\mathcal{N}(c, a_1,...,a_n) : c\in N\}$$

By Definition~\ref{equivalencia e inmersi\'on elemental} (2), we need to do an inductive argument on $\mathbf{L}$-formulas in order to prove $\mathcal{M}\preccurlyeq\mathcal{N}$. It is straightforward to see that $\mathcal{M}\subseteq \mathcal{N}$  guarantees the basic cases, and from the hypothesis if follows the inductive step by using connectives and the quantifier $\bigwedge$, so we have just to check the inductive step by using $\bigvee$. Let $\varphi(x,x_1,...,x_n)$ be an {\bf L}-formula and $a_1,...,a_n\in M$. Let $b\in \mathbf{V}$ be a dualizing element. Notice that by Remark~\ref{b-x mapping} $b\dotdiv \square$ is uniformly continuous in the symmetric topology and so it is a connective. Theferore,

\begin{eqnarray*}
(\bigvee_x\varphi(x,x_1,...,x_n))^\mathcal{M}(a_1,...,a_n)&=& \bigvee\{\varphi(c,a_1,...,a_n) : c\in M\} \\
&&(\text{by Definition~\ref{interpretaci\'on de sentencias} (4)} )\\
&=&\bigvee\{b\dotdiv(b\dotdiv\varphi^\mathcal{M}(c,a_1,...,a_n)) : c\in M\} \\ &&(\text{$b$ is a dualizing element)} \\
&=& b\dotdiv \bigwedge\{b\dotdiv\varphi^\mathcal{M}(c,a_1,...,a_n) : c\in M\} \\ &&  (\text{by Proposition~\ref{residuar un \'infimo}})\\
&=& b\dotdiv \bigwedge\{b\dotdiv\varphi^\mathcal{N}(c,a_1,...,a_n) : c\in N\} \\ &&
(\text{hypothesis induction on $\varphi$} \\
&& \text{and by applying this statement to}\\
&& \text{$b\dotdiv \varphi(x,x_1,...,x_n)$ and}\\
&& \text{$\bigwedge_x (b\dotdiv \varphi(x,x_1,...,x_n)$} \\
&& \text{-$b\dotdiv\square$ is a connective by Remark~\ref{b-x mapping}-)}\\
&=&\bigvee\{b\dotdiv(b\dotdiv\varphi^\mathcal{N}(c,a_1,...,a_n)) : c\in N\} \\ &&(\text{by Proposition~\ref{residuar un \'infimo}})\\
&=& \bigvee\{\varphi^\mathcal{N}(c,a_1,...,a_n) : c\in N\} \\ 
&&\text{($b$ is a dualizing element)} \\
&=&(\bigvee_x\varphi(x,x_1,...,x_n))^\mathcal{N}(a_1,...,a_n)  \\ 
&& (\text{by Definition~\ref{interpretaci\'on de sentencias}})
\end{eqnarray*}
\end{proof}

\section{$D$-products and \L o\'s Theorem in co-quantale valued logics.}

Chang and Keisler (\cite{ChKe66} ) defined some logics with truth values on Hausdorff compact topological spaces. In that context, they provided a version of \L os' Theorem, which implies a Compactness Theorem in their logic and the existence of saturated models (as it holds in first order logic). This approach is rediscovered in~\cite{BeBeHeUs}, but by taking the particular case of truth values in the unit interval $[0,1]$. We propose to generalize the version of \L o\'s Theorem in our context of value co-quantale valued logics, as a test question of the logics proposed in this paper.

\subsubsection{$D$-limits}
Let us fix ${\bf V}$ a $\mathbf{V}$-domain value co-quantale provided with its symmetric topology. By Remark~\ref{Domain-Hausdorff-Compact}, $(\mathbf{V},\tau^s)$ is compact and Hausdorff, therefore we may apply Lemma~\ref{lema de convergencia de ultrafiltros} to the symmetric topology of $\mathbf{V}$. Let $I$ be a non empty set and $D$ an ultrafilter over $I$.

\begin{remark}\label{simetr\'ia en d_V}
We need to assume that ${\bf V}$ is provided with its symmetric distance $d^s_V(p,q):=(p\dotdiv q) \vee (q\dotdiv p)$, because we need to guarantee that  $p\dotdiv q\leq d^s_V(p,q)$, which might fail for the original distance $d(p,q):=q\dotdiv p$ in $\mathbf{V}$. The inequality $q\dotdiv p\leq d(p,q)$ always holds for both the original and the symmetric distances of $\mathbf{V}$.
\\ \\
\indent From now, for the sake of simplicity, let us denote the symmetric distance of $\mathbf{V}$ by $d_V$.
\end{remark}


The following is a very known fact about convergence of sequences in Hausdorff Compact topological spaces. 

\begin{lem}\label{lema de convergencia de ultrafiltros}(\cite{ChKe66}; Thrm 1.5.1.)
If $(X,\tau)$ is a Hausdorff compact topological space, given a sequence $(x_i)_{i\in I}$ in $X$  there exists a unique $x\in X$ such that for any neighborhood $V$ of $x$, then $\{i\in I| x_i\in V\}\in D$.
\end{lem}


Fact~\ref{SisFundeVecindades} and Lemma~\ref{lema de convergencia de ultrafiltros} allow us to give the following notion of convergence in our setting.

\begin{defi}\label{D-ultral\'imite}
Given $(a_i)_{i\in I}$ a sequence in ${\bf V}$, the unique $a\in {\bf V}$ which satisfies that for any $\epsilon\in{\bf V}^+$ we have that $\{i\in I| d_V(a,a_i)\leq\epsilon\}\in D$ is said to be the {\bf $D$-ultralimit} of the sequence $(a_i)_{i\in I}$, which we denote it by $lim_{i,D}a_i$.
\end{defi}

\begin{defi}\label{Conjuntos grandes en vecindades del l\'imite}
Given $\epsilon\in{\bf V}^+$, define $$A(\epsilon):=\{j\in I| d_V(lim_{i,D}a_i,a_j)\leq\epsilon\}.$$
\end{defi}

\begin{Prop}\label{acotar D-l\'imites}
Let $(a_i)_{i \in I}$ be a sequence in ${\bf V}$ and $b\in{\bf V}$.
\begin{enumerate}
    \item If there exists $A\in D$ such that for all $j\in A$ we have that $b\leq a_j$, then $b\leq lim_{i,D}a_i$.
    \item If $d_V$ is the symmetric distance of $\mathbf{V}$ and there exists $A\in D$ such that for all $j\in A$ we have that $b\geq a_j$, then $b\geq lim_{i,D}a_i$. 
\end{enumerate}
\end{Prop}
\begin{proof}
\begin{enumerate}
    \item It is enough to prove that for any $\epsilon\in {\bf V}$ such that $0\prec\epsilon$ we have that $ b \dotdiv lim_{i,D}a_i\leq\epsilon$, because $\mathbf{V}$ is co-continuous (Definition~\ref{aproximaci\'on}) and then we would have that $b \dotdiv lim_{i,D}a_i\leq\bigwedge\{\epsilon \in{\bf V} : 0\prec\epsilon\}=0$ and by Proposition~\ref{adjunci\'on} (4) $b\leq lim_{i, D}a_i$ holds, as desired.\\
Let $\epsilon\in{\bf V}^+$, so by definition of $\lim_{i,D} a_i$  we know that $\{j\in I : d_V(lim_{i,D}a_i, a_j)\leq\epsilon\}=:A(\epsilon)\in D$.
By hypothesis $A\in D$, therefore $A(\epsilon)\cap A\in D$ and so there exists $j\in A(\epsilon)$ such that $b\leq a_j$ (because $j\in A$). Notice that $a_j\dotdiv lim_{i,D}a_i\leq d_V(lim_{i,D}a_i, a_j)\le \epsilon$ (by Remark~\ref{simetr\'ia en d_V} and since $j\in A(\epsilon)$).
By Lemma~\ref{monoton\'ia-intercambio} and since $b\leq a_j$, we may say that $b\dotdiv lim_{i,D}a_i\leq a_j\dotdiv lim_{i,D}a_i \leq\epsilon$
\item It is enough to prove that whenever $0\prec\epsilon$ we have that $lim_{i,D}a_i\dotdiv b\leq \epsilon$. As above, $A(\epsilon)\in D$. Since $A\in D$, there exists $j\in A(\epsilon)$ such that $a_j\leq b$, since $d_V$ is the symmetric distance of $\mathbf{V}$ and by Remark~\ref{simetr\'ia en d_V} we have that
$lim_{i,D}a_i\dotdiv a_j\leq d_V(lim_{i,D}a_i,a_j):=(lim_{i,D}a_i\dotdiv a_j)\vee(a_j\dotdiv lim_{i,D}a_i)\leq\epsilon$; by Lemma~\ref{monoton\'ia-intercambio} and since $a_j\leq b$ we have that $lim_{i,D}a_i\dotdiv b\leq lim_{i,D}a_i\dotdiv a_j\leq d_V(lim_{i,D}a_i,a_j)\leq\epsilon$, as desired.
\end{enumerate}
\end{proof}

The following fact is a kind of converse of the previous result,  by assuming co-divisibility (Definition~\ref{Cuantal Co-divisible}).

\begin{Prop}\label{cota fuerte}
Suppose that ${\bf V}$ is co-divisible (Definition~\ref{Cuantal Co-divisible}) and that $d_V$ is the symmetric distance of $\mathbf{V}$. Let $(a_i)_{i\in I}$ be a sequence in ${\bf V}$ and $b\in {\bf V}$ such that $lim_{i,D}a_i\leq b$ and $0\prec b\dotdiv lim_{i,D}a_i$. Therefore, there exists $A\in D$ such that $i\in A$, $a_i\leq b$.
\end{Prop}
\begin{proof}
By hypothesis $0\prec b\dotdiv lim_{i,D}a_i$, therefore by Lemma~\ref{densidad} there exists some $\epsilon\in{\bf V}$ such that $0 \prec \epsilon\prec b\dotdiv lim_{i,D}a_i$.  By Lemma~\ref{propiedades b\'asicas de la relaci\'on prec} (1) we may say $0 \prec \epsilon\leq b\dotdiv lim_{i,D}a_i$. By taking $A(\epsilon):=\{j\in I : d_V(lim_{i,D}a_i,a_j)\leq\epsilon\}$, we have that $A(\epsilon)\in D$ (by definition of $lim_{i,D}a_i$). By Remark~\ref{simetr\'ia en d_V} and definition of $A(\epsilon)$, for all $j\in A(\epsilon)$ we have that $a_j \dotdiv lim_{i,D}a_i\leq d_V(lim_{i,D}a_i,a_j)\leq\epsilon\leq b\dotdiv lim_{i,D}a_i$. By Proposition~\ref{adjunci\'on} (1) and Lemma~\ref{caracterizaci\'on de co-divisibilidad} (by hypothesis, $\mathbf{V}$ is co-divisible), we may say that $a_j\leq (b\dotdiv lim_{i,D}a_i) + lim_{i,D}a_i=b$, so $A:=A(\epsilon)$ is the required set.
\end{proof}


\begin{lem}
\label{ser de Cauchy discreto}
Let $K\neq \emptyset$ a set and $(a_k)_{k\in K}$ bea $K$-sequence in $X$. Therefore,  $\bigwedge_{k\in K}(\bigvee_{l\in K}(a_l\dotdiv a_k))=0$, if and only if, for all $\epsilon\in{\bf V}^+$ there exists $k\in K$ such that $\bigvee_{l\in K}a_l\dotdiv \epsilon\leq a_k$. Also, $\bigwedge_{k\in K}(\bigvee_{l\in K}(a_k\dotdiv a_l))=0$, if and only if, for all $\epsilon\in{\bf V}^+$ there exists $k\in K$ such that $a_k\dotdiv\epsilon\leq\bigwedge_{l\in K}a_l$.
\end{lem}
\begin{proof}
Suppose that $\bigwedge_{k\in K}(\bigvee_{l\in K}(a_l\dotdiv a_{k}))=0$ and let $\epsilon\in{\bf V}^+$ (i.e., $0\prec \epsilon$), by definition of $\prec$ and since $0\geq\bigwedge_{k\in K}(\bigvee_{l\in K}(a_l\dotdiv a_{k}))$ we may say that there exists $k\in K$ such that $\bigvee_{l\in K}(a_l\dotdiv a_{k})\leq\epsilon$. By Fact~\ref{resta truncada es adjunto izquierdo},  $\bigvee_{l\in K}(a_l\dotdiv a_{k})=(\bigvee_{l\in K}a_l)\dotdiv a_{k}$, so $(\bigvee_{l\in K}a_l)\dotdiv a_{k}\leq\epsilon$, and by Proposition~\ref{adjunci\'on} we have that $(\bigvee_{l\in K}a_k)\dotdiv\epsilon\leq a_k$.\\
Conversely, suppose that for all $\epsilon\in{\bf V}^+$ there exists $k\in K$ such that $\bigvee_{l\in K}a_l\dotdiv \epsilon\leq a_k$. By Proposition~\ref{adjunci\'on} (1) and Fact~\ref{resta truncada es adjunto izquierdo} we may say that $\bigvee_{l\in K}(a_l \dotdiv a_k)= (\bigvee_{l\in K}a_l)\dotdiv a_k\leq\epsilon$. Therefore (by Proposition~\ref{continuidad}),  $\bigwedge_{k\in K}(\bigvee_{l\in K}(a_l\dotdiv a_k))\leq\bigwedge\{\epsilon\in \mathbf{V} :  0\prec\epsilon\}=0$

A similar idea works for proving the second statement. Suppose $\bigwedge_{k\in K}(\bigvee_{l\in K}(a_k\dotdiv a_l))=0$ and let $\epsilon\in{\bf V}^+$ (i.e., $0\prec \epsilon$). By definition of $\prec$ there exists $k\in K$ such that $\bigvee_{l\in K}(a_k\dotdiv a_l)\leq\epsilon$, and by Proposition~\ref{residuar un \'infimo} we may say that $a_k\dotdiv\bigwedge_{l\in K}a_l\leq\epsilon$. By Proposition~\ref{adjunci\'on} (1), this implies that  $a_k\dotdiv\epsilon\leq\bigwedge_{l\in K}a_l$.\\
Conversely, suppose that for all $\epsilon\in\mathbf{V}^+$ exist $k\in K$ such that $a_k\dotdiv\epsilon\leq\bigwedge_{l\in K}a_l$, so by Proposition~\ref{adjunci\'on} (1) and Proposition~\ref{residuar un \'infimo} we have that $\bigvee_{l\in K}(a_k\dotdiv a_l)=a_k\dotdiv\bigwedge_{l\in K }a_l\leq \epsilon$, therefore (by Proposition~\ref{continuidad}) $\bigwedge_{k\in K}(\bigvee_{l\in K}(a_k\dotdiv a_l))\le  \bigwedge\{\epsilon\in \mathbf{V}: 0\prec\epsilon\}=0$.
\end{proof}

\begin{remark}
The following fact is very important to deal with the quantifier cases in the proof of \L o\'s Theorem in this setting. The idea of these proofs is quite similar to the one presented in \cite{BeBeHeUs}, but adapted to our setting.
\end{remark}

\begin{Prop}\label{teorema de cuantificadores}(c.f.\cite{BeBeHeUs}; Lemma 5.2) Let $d_{V}$ be the symmetric distance of $\mathbf{V}$. Let $S\neq \emptyset$ and $(F_i)_{i\in I}$ be a sequence of mappings with domain $S$ and codomain ${\bf V}$, then: \\
$\bigwedge\{lim_{i, D}F_i(x): x \in S\}\geq lim_{i,D}(\bigwedge\{F_i(x):x\in S\})$ and\\
$\bigvee\{lim_{i, D}F_i(x): x \in S\}\leq lim_{i,D}(\bigvee\{F_i(x):x\in S\})$. 
\\
Also, given $\epsilon\in{\bf V}^+$ there exist sequences $(b_i)_{i \in I} $ and $ (c_i)_{i\in I}$ in $S$ such that\\
$lim_{i, D}F_i(b_i) + \epsilon \geq lim_{i,D}(\bigvee\{F_i(x): x \in S\})$ and \\$lim_{i, D}F_i(c_i) \dotdiv \epsilon \leq lim_{i,D}(\bigwedge\{F_i(x): x \in S\})$, whenever there exist $B,C\in D$ such that  $\bigwedge_{y\in S}(\bigvee_{x\in S}(F_i(x)\dotdiv F_i(y)))=0$ for all $i\in B$ and $\bigwedge_{y\in S}(\bigvee_{x\in S}(F_i(y)\dotdiv F_i(x)))=0$ for all $i\in C$.
\end{Prop}

\begin{proof}
Let $r_i:=\bigwedge\{F_i(x) : x\in S\}$, $r:=lim_{i,D}r_i$ and $\epsilon \in {\bf V}^+$. Define $A(\epsilon)=\{j\in I : d_V(r,r_j)\leq\epsilon\}$, so by Definition of $lim_{i,D}r_i$ we may say that $A(\epsilon)\in D$. Notice that if $j\in A(\epsilon)$ and by Remark~\ref{simetr\'ia en d_V} we have that $r\dotdiv r_j\leq d_V(r,r_j)\leq\epsilon$. So, by Proposition~\ref{adjunci\'on} (1) it follows that $r\leq r_j+\epsilon$ and then $r\dotdiv \epsilon\leq r_j$.

Let $x\in S$, then $r\dotdiv \epsilon\leq r_j:=\bigwedge\{F_j(y) : y\in S\}\leq F_j(x)$. Since $A(\epsilon)\in D$, by Proposition~\ref{acotar D-l\'imites} (1) we have that $r\dotdiv\epsilon\leq lim_{i,D}F_i(x)$, and since it holds for any $x\in S$ then $r\dotdiv\epsilon\leq\bigwedge\{lim_{i,D}F_i(x) : x\in S\}$. By Proposition~ \ref{adjunci\'on} (1) and by commutativity of $+$, it follows that $r\leq\bigwedge\{lim_{i,D}F_i(x) : x\in S\} + \epsilon $. Since $0\prec \epsilon$ was taken arbitrarily, by Fact~\ref{descender con la suma} we may say that $r\leq\bigwedge\{lim_{i,D}F_i(x)| x\in S\}$.
\\ \\
Let $s_i:=\bigvee\{F_i(x) : x\in S\}$, $s:=lim_{i,D}s_i$ and $\epsilon\in{\bf V}^+$. By taking $A^{\prime}(\epsilon)=\{j\in I : d_V(s,s_j)\leq\epsilon\}$, then $A^{\prime}(\epsilon)\in D$ (definition of $lim_{i,D}s_i$). Given $j\in A^{\prime}(\epsilon)$, by Remark~\ref{simetr\'ia en d_V} we have that $s_j\dotdiv s\leq d_V(s,s_j)\leq\epsilon$ and by Proposition~\ref{adjunci\'on} (1) we may say that $s_j\leq s +\epsilon$. Given $x\in S$, it follows that $F_j(x)\leq s_j\leq s+\epsilon$. Since $d_V$ is symetric, by Proposition~\ref{acotar D-l\'imites} (2) and since $A^{\prime}(\epsilon)\in D$, we have that $lim_{i,D}F_i(x)\leq r+ \epsilon$; since  $x\in S$ was taken arbitrarily, then $\bigvee\{lim_{i,D}F_i(x) | x\in S\}\leq r+ \epsilon$. Since $0\prec\epsilon$ is arbitrary, by Fact~\ref{descender con la suma} we have that $\bigvee\{lim_{i,D}F_i(x) | x\in S\}\leq r$.
\\ \\
Let us continue with the proof of the two last facts. Let $\epsilon\in{\bf V}^+$, so by Lemma~\ref{argumento epsilon medios} there exists some $\theta\in {\bf V}^+$ such that $0\prec\theta$ y $\theta + \theta \leq\epsilon$. By hypothesis, suppose there exists $B\in D$ such that $\bigwedge_{y\in S}(\bigvee_{x\in S}(F_i(x)\dotdiv F_i(y)))=0$ for all $i\in B$. Let $i\in B$, since $0\prec\theta$ and by Lemma~\ref{ser de Cauchy discreto} take $b_i\in S$ such that $\bigvee_{x\in S}F_i(x)\dotdiv \theta\leq F_i(b_i)$, and so by Proposition~\ref{adjunci\'on} we may say that $\bigvee_{x\in S}F_i(x)\dotdiv F_i(b_i)\leq \theta$. If $i\notin B$, choose $b_i$ as any element in $S$.  Let $B(\theta):=\{i\in I : d_V(lim_{i,D}F_i(b_i),F_i(b_i))\leq\theta\}$, so by definition of $lim_{i,D}F_i(b_i)$ we know that $B(\theta)\in D$. Therefore, if $j\in B\cap B(\theta)$ then $\bigvee_{x\in S}F_j(x)\leq lim_{i,D}F_i(b_i)+\epsilon$. In fact, if $j\in B \cap B(\theta)$ then
\begin{eqnarray*}
\bigvee_{x\in S}F_j(x)\dotdiv lim_{i\in I}F_i(b_i)&\leq& d_V(\bigvee_{x\in S}F_j(x),lim_{i\in I}F_i(b_i))\\
&& \text{(by Remark~\ref{simetr\'ia en d_V})}\\
&\leq& d_V(\bigvee_{x\in S}F_j(x), F_j(b_j)) + d_V(F_j(b_j),lim_{i\in I}F_i(b_i))\\
&=& \left[\left(\bigvee_{x\in S}F_j(x)\dotdiv F_j(b_j)\right) \vee \left(F_j(b_j)\dotdiv \bigvee_{x\in S}F_j(x)\right)\right] \\
&&+ d_V(F_j(b_j),lim_{i\in I}F_i(b_i))\\
&=& \bigvee_{x\in S}F_j(x)\dotdiv F_j(b_j) + d_V(F_j(b_j),lim_{i\in I}F_i(b_i))\\
&\leq& \theta + \theta\\
&\leq& \epsilon.
\end{eqnarray*}
 Therefore, $\bigvee_{x\in S}F_j(x)\dotdiv lim_{i\in I}F_i(b_i)\leq\epsilon$, and by Proposition~\ref{adjunci\'on} (1) we may say that $\bigvee_{x\in S}F_j(x)\leq lim_{i\in I}F_i(b_i) + \epsilon$. Hence, since $B\cap B(\theta)\in D$ by Proposition~\ref{acotar D-l\'imites} (2) we have that $lim_{i,D}(\bigvee\{F_i(x): x \in S\})\leq lim_{i, D}F_i(b_i) + \epsilon$.
 \\ \\
Let us construct the sequence $(c_i)_{i\in I}$ as follows: Let $\epsilon\in{\bf V}^+$. By hypothesis and Lemma~\ref{ser de Cauchy discreto}, there for all $j\in C$ there exists $c_j\in S$ such that $F_j(c_j)\dotdiv\theta\leq\bigwedge_{x\in S}F_j(x)$, where $\theta\in{\bf V}^+$ satisfies $\theta+\theta\leq\epsilon$. By Proposition~\ref{adjunci\'on} (1), it follows that $F_j(c_j)\dotdiv\bigwedge_{x\in S}F_j(x)\leq\theta$; since $\bigwedge_{x\in S}F_j(x)\leq\ F_j(c_j)$, it implies that $d_V(F_j(c_j),\bigwedge_{x\in S}F_j(x))\leq\theta$. If $j\notin C$, take $c_j$ as any element of $S$. Let $C(\theta):=\{i\in I : d_V(lim_{i\in I}F_i(c_i),F_i(c_i))\leq\theta\}$. Therefore, if $j\in C\cap C(\theta)$ we have that 
\begin{eqnarray*}
lim_{i,D}F_i(c_i)\dotdiv\bigwedge_{x\in S}F_j(x) &\leq& d_V(lim_{i,D}F_i(c_i),\bigwedge_{x\in S}F_j(x))\\
&\leq& d_V(lim_{i,D}F_i(c_i),F_j(c_j))+d_V(F_j(c_j),\bigwedge_{x\in S}F_j(x))\\
&\leq& \theta+\theta\\
&\leq&\epsilon
\end{eqnarray*}
By Proposition~\ref{adjunci\'on} (1),  we have that $lim_{i,D}F_i(c_i)\dotdiv\epsilon\leq\bigwedge_{x\in S}F_j(x)$. Since $C\cap C(\theta)\in D$ and by Proposition~\ref{acotar D-l\'imites} (1), $lim_{i, D}F_i(c_i) \dotdiv \epsilon \leq lim_{i,D}(\bigwedge\{F_i(x): x \in S\})$.
\end{proof}

\subsubsection{$D$-product and $D$-ultraproduct of spaces and mappings}
%

\begin{Prop} \label{paso al cociente} Let $d_V$ be the symmetric distance of $\mathbf{V}$. If $(M_i)_{i\in I}$ is a sequence of $\mathbf{V}$continuity spaces such that for all $i\in I$ all distances $d_{M_i}$ are symmetric, then in the cartesian product $\prod_{i\in I}M_i$, the relation $\sim$ defined by $(x_i)_{i\in I}\sim (y_i)_{i\in I}$, if and only if,  $lim_{i,D}d_{M_i}(x_i,y_i)=0$, is an equivalence relation.
\end{Prop}
\begin{proof}
Reflexivity follows trivially and symmetry follows from symmetry of all $d_{M_i}$. Let us focus on the transitivity. Let $(x_i)_{i\in I}, (y_i)_{i\in I}, (z_i)_{i\in I}\in \prod_{i\in I}M_i$ be such that $lim_{i,D}d_{M_i}(x_i,y_i)=0$ and $lim_{i,D}d_{M_i}(y_i,z_i)=0$. Since $\mathbf{V}$ is a value co-quantale, it is enough to check that for any $\epsilon\in{\bf V}^+$ we have that  $lim_{i,D}d_{M_i}(x_i,z_i)\leq\epsilon$. Let $\epsilon\in{\bf V}^+$, so by Lemma~\ref{argumento epsilon medios} there exists $\theta\in{\bf V}^+$ such that $\theta+\theta\leq\epsilon$. Since $0\prec\theta$ and by definition of $lim_{i,D}d_{M_i}(x_i,y_i)$ and $lim_{i,D}d_{M_i}(y_i,z_i)$,  $A(\theta):=\{i\in I : d_V(lim_{i,D}d_{M_i}(x_i,y_i),d_{M_i}(x_i,y_i)\leq\theta\}$ and $B(\theta):=\{i\in I : d_V(lim_{i,D}d_{M_i}(y_i,z_i),d_{M_i}(y_i,z_i)\leq\theta\}$ belong to $D$. By hypothesis, $lim_{i,D}d_{M_i}(x_i,y_i)=0$ y $lim_{i,D}d_{M_i}(y_i,z_i)=0$. Notice that by Proposition~\ref{adjunci\'on} (4) and since $0=\min \textsc{V}$, it follows that
\begin{eqnarray*}
d_V(lim_{i,D}d_{M_i}(x_i,y_i),d_{M_i}(x_i,y_i))\\
&=&\vee\{lim_{i,D}d_{M_i}(x_i,y_i)\dotdiv d_{M_i}(x_i,y_i),\\
&&d_{M_i}(x_i,y_i)\dotdiv lim_{i,D}d_{M_i}(x_i,y_i)\}\\
&& \text{($d_{M_i}$ is assumed to be symmetric)}\\
&=&\vee\{0\dotdiv d_{M_i}(x_i,y_i), d_{M_i}(x_i,y_i)\dotdiv 0\}\\
&=&\vee\{0,d_{M_i}(x_i,y_i)\}\\
&=& d_{M_i}(x_i,y_i)
\end{eqnarray*}
 Therefore, $A(\theta)=\{i\in I : d_{M_i}(x_i,y_i)\leq\theta \}$ and $B(\theta)=\{i\in I : d_{M_i}(y_i,z_i)\leq\theta \}$. So, if $i\in A(\theta)\cap B(\theta)$ then $d_{M_i}(x_i,z_i)\leq d_{M_i}(x_i,y_i)+d_{M_i}(y_i,z_i)\leq\theta+\theta\leq\epsilon$. Since $A(\theta)\cap B(\theta)\in D$, by Lemma~\ref{acotar D-l\'imites} (2) we may say that  $lim_{i,D}d_{M_i}(x_i,z_i)\leq\epsilon$.
\end{proof}

\begin{remark}
 In general, we do not require that all continuity spaces in the sequence  $(M_i)_{i\in I}$ are symmetric, where in that case $\sim$ might not be an equivalence relation. We just need this requirement if $(M_i,d_{M_i}):=(V,d_M)$ for all $i\in I$. As we will see in the following proposition, we can provided a continuity space structure to the cartesian product $\prod_{i\in I}M_i$, without assuming the symmetry on $d_{M_i}$.
\end{remark}

\begin{Prop}\label{la distancia en un D-producto}
Suppose that $d_V$ is the symmetric distance. If $(M_i)_{i\in I}$ is a sequence of $\mathbf{V}-continuity spaces$, $(\prod_{i\in I}M_i,d_D)$ is a $\mathbf{V}$-continuity space, where $d_D:\prod_{i\in I}M_i\times \prod_{i\in I}M_i\rightarrow{\bf V}$ is defined by $((x_i)_{i\in I},(y_i)_{i\in I})\mapsto lim_{i,D}d_{M_i}(x_i,y_i)$.
\end{Prop}
\begin{proof}
Given $(x_i)_{i\in I}\in\prod_{i\in I}M_i$, then 
\begin{eqnarray*}
d_{\prod_{i\in }}((x_i)_{i\in I},(x_i)_{i\in I}) &=& lim_{i,D}d_{M_i}(x_i,x_i)\\
& =&lim_{i,D}0\\
&=&0
\end{eqnarray*}

Let $(x_i)_{i\in I}, (y_i)_{i\in I}, (z_i)_{i\in I} \in \prod_{i\in I}M_i$, $a_i:=d_{M_i}(x_i,y_i)$, $b_i:=d_{M_i}(x_i,z_i)$, $c_i:=d_{M_i}(y_i,z_i)$,  $a:=lim_{i,D}d_{M_i}(x_i,y_i)$, $b:=lim_{i,D}d_{M_i}(x_i,z_i)$ and  $c:=lim_{i,D}d_{M_i}(y_i,z_i)$. We want to see that $a\leq b + c$, which by  Proposition~\ref{adjunci\'on} (1) it is enough to prove that $a\dotdiv b\leq c$. 
Let $\epsilon\in{\bf V}^+$, so by Corollary~\ref{argumentos epsilon sobre n} there exists some $\theta\in{\bf V}^+$ such that $\theta+\theta+\theta\prec\epsilon$. By Lemma~\ref{propiedades b\'asicas de la relaci\'on prec} (1),  $\theta+\theta+\theta\leq\epsilon$. By definition of $lim_{i,D}d_{M_i}(x_i,y_i)=:a$, $lim_{i,D}d_{M_i}(x_i,z_i)=:b$ and  $lim_{i,D}d_{M_i}(y_i,z_i)=:c$, $A:=\{i\in I : d_V(a,a_i)\leq\theta\}$, $B:=\{i\in I : d_V(b,b_i)\leq\theta\}$ and $C:=\{i\in I : d_V(c,c_i)\leq\theta\}$ belong to $D$. So, $A\cap B \cap C\in D$. Let $i\in A\cap B\cap C$, so by Proposition~\ref{adjunci\'on} (1) and (2) we may say that $(a\dotdiv b)\leq(a\dotdiv a_i)+(a_i\dotdiv b_i)+(b_i\dotdiv b)$. Since $d_{M_i}$ satisfies transitivity, by Proposition~~\ref{adjunci\'on} (1) we have that $a_i\dotdiv b_i\leq c_i$. Since $i\in C$, we may say that $c_i\dotdiv c\leq \vee\{c_i\dotdiv c, c\dotdiv c_i\}=d_V(c,c_i)\leq\theta$. By Proposition~\ref{adjunci\'on} (1), it follows that $c_i\leq c+\theta$. By monotonicity, $a\dotdiv b\leq (a\dotdiv a_i)+(b_i\dotdiv b)+(c+\theta)$. Since $i\in A(\theta)\cap B(\theta)$, we also have that $a\dotdiv a_i\leq \vee\{a\dotdiv a_i,a_i\dotdiv a\}=d_V(a,a_i)\leq\theta$ and $b_i\dotdiv b\leq\vee\{b_i\dotdiv b, b\dotdiv b_i\}=d_V(b,b_i)\leq\theta$. Hence, $a\dotdiv b\leq (a\dotdiv a_i)+(b_i\dotdiv b)+(c+\theta)\leq \theta + \theta + (c+\theta)=\theta+\theta+\theta+c\leq \epsilon+c$. Since $\epsilon\in{\bf V}^+$ was chosen arbitrarily, it follows that $a\dotdiv b\leq c$.
\end{proof}

\begin{defi}\label{D-producto de espacios}
We call {\bf D-product} of the sequence of $\mathbf{V}$-continuity spaces $(M_i,d_{M_i})_{i\in I}$ to the $\mathbf{V}$-continuity space $(M_D,d_{M_D}):=(\prod_{i\in I}M_i,d_D)$ defined in Proposition~\ref{la distancia en un D-producto}.
\end{defi}

\begin{lem}\label{el cociente es un espacio}
Let $d_V$ be the symmetric distance of $\mathbf{V}$. If $(M_i,d_{M_i})$ is a sequence of $\mathbf{V}$-continuity spaces such that $d_{M_i}$ is symmetric for all $i\in I$, then $\left( \prod_{i\in I}M_i/\sim, \overline{d}_D\right)$ is a $\mathbf{V}$-continuity space, provided that $\overline{d}_D$ is defined by $([(x_i)_{i\in I}],[(y_i)_{i\in I}])\mapsto lim_{i,D}d_{M_i}(x_i,y_i)$, where $\sim$ is defined as in Proposition~\ref{paso al cociente}.
\end{lem}
\begin{proof}
First, let us check that $((x_i)_{i\in I},(y_i)_{i\in I})\mapsto lim_{i,D}d_{M_i}(x_i,y_i)$ is well-defined. Let  $(a_i)_{i\in I}, (b_i)_{i\in I},(c_i)_{i\in I}, (d_i)_{i\in I}\in \prod_{i\in I}M_i$ be such that $(a_i)_{i\in I}\sim (b_i)_{i\in I}$ and $(c_i)_{i\in I}\sim (d_i)_{i\in I}$. By Fact~\ref{SisFundeVecindades} and the definition of {D-ultralimits}, in order to prove that $lim_{i,D}d_{M_i}(a_i,c_i)=lim_{i,D}d_{M_i}(b_i,d_i)$ it is enough to prove that for every $\epsilon\in{\bf V}^+$ we have that
\begin{center}
$\{ i\in I : d_V(lim_{i,D}d_{M_i}(a_i,c_i),d_{M_i}(b_i,d_i))\leq\epsilon\}\in D$.
\end{center}
Let $\epsilon\in{\bf V}^+$, so by Proposition~\ref{argumento epsilon medios} there exists $\theta\in{\bf V}^+$ such that $\theta+\theta \prec \epsilon$. By Lemma~\ref{propiedades b\'asicas de la relaci\'on prec} (1), we may say that $\theta+\theta \leq\epsilon$. In similar way we can prove that there exists some $\delta\in{\bf V}^+$ such that $\delta+\delta\leq\theta$. Since $(a_i)_{i\in I}\sim (b_i)_{i\in I}$, by definition of $\sim$  we have that $lim_{i,D}d_{M_i}(a_i,b_i)=0$, hence  for all $\gamma\in{\bf V}^+$ we have that $\{i\in I: d_V(lim_{i,D}d_{M_i}(a_i,b_i),d_{M_i}(a_i,b_i))\leq\gamma\}\in D$. By Remark~\ref{la distancia al 0 en el caso sim\'etrico}, we may say that $\{i\in I: d_{M_i}(a_i,b_i)\leq\gamma\}\in D$. Since $\delta \in {\bf V}^+$, in particular $A:=\{i\in I: d_{M_i}(a_i,b_i)\leq\delta\}\in D$. Analogously, since $(c_i)_{i\in I}~(d_i)_{i\in I}$ we may say that $B:=\{i\in I: d_{M_i}(c_i,d_i)\leq\delta\}\in D$. So, if $i\in A\cap B$ then 
\begin{eqnarray*}
d_{M_i}(a_i,c_i)&\leq& d_{M_i}(a_i,b_i)+d_{M_i}(b_i,d_i)+d_{M_i}(d_i,c_i)\\
&& (\text{by transitivity of } \ d_{M_i})\\
&\leq& \delta + d_{M_i}(b_i,d_i) + \delta \\
&& (\text{since} \ i\in A\cap B)\\
&\leq& d_{M_i}(b_i,d_i) + \theta \\
&&(\text{by Proposition~\ref{monotonia de la suma} and since} \ \delta+\delta \leq\theta )
\end{eqnarray*}
In a similar way we may say that $d_{M_i}(b_i,d_i)\leq d_{M_i}(a_i,c_i) + \theta$, whenever  $i\in A\cap B$. By Proposition~\ref{adjunci\'on} (1),  $d_{M_i}(b_i,d_i)\dotdiv d_{M_i}(a_i,c_i)\leq \theta$ and $d_{M_i}(a_i,c_i)\dotdiv d_{M_i}(b_i,d_i)\leq \theta$, therefore $d_V(d_{M_i}(a_i,c_i),d_{M_i}(b_i,d_i))\leq\theta$. Notice that by definition of $lim_{i,D}d_{M_i}(a_i,c_i)$, $A(\theta):=\{i\in I : d_V(lim_{i,D}d_{M_i}(a_i,c_i),d_{M_i}(a_i,c_i))\leq\theta\}\in D$. So, if $i\in A\cap B\cap A(\theta)$ we have that
\begin{eqnarray*}
d_V(lim_{i,D}d_{M_i}(a_i,c_i),d_{M_i}(b_i,d_i)) &\leq&  d_V(lim_{i,D}d_{M_i}(a_i,c_i),d_{M_i}(a_i,c_i))\\
&&+d_V(d_{M_i}(a_i,c_i),d_{M_i}(b_i,d_i))\\
&&(\text{by transitivity of } d_V)\\
&\leq& \theta + \theta \\ && (\text{since } \ i\in A\cap B\cap A(\theta) )\\
&\leq& \epsilon \\
&& (\text{by Proposition~\ref{monotonia de la suma} and since} \\
&&\theta+\theta \leq\epsilon )
\end{eqnarray*}
\end{proof}

\begin{defi}\label{UltradeEspacios}\index{Ultraproducto!de espacios}
Let $(M_i,d_{M_i})_{i\in I}$ be a sequence of $\mathbf{V}$-continuity spaces.
The $\mathbf{V}$-continuity space $\prod_{i\in I}M_i/\sim$ provided with the symmetric distance $d_{M_D}([(x_i)_{i\in I}],[(y_i)_{i\in I}]):=lim_{i,D}d_{M_i}(x_i,y_i)$ is called the $D$-{\bf ultraproduct} of the sequence $(M_i,d_{M_i})_{i\in I}$.
\end{defi}

\begin{defi}(c.f. \cite{BeBeHeUs}; pag 25)\label{UltradeFunciones}
Let $(M_i)_{i \in I}$, $(N_i)_{i \in I}$ be sequences of $\mathbf{V}$-continuity spaces and $K\in{\bf V}$ be such that $K\in \mathbf{V}$ greater than the diameter of all the considered spaces. Given a fixed $n\in\mathbb{N}\setminus\{0\}$ and $(f_i:M^n_i \rightarrow N_i)_{i\in I}$ a sequence of uniformly continuous mappings provided with the same modulus of uniform continuity. The mapping $f_D:(M_D,d_D^M)^n\rightarrow(N_D,d_D^N)$ defined by 
\begin{center}
    $((x^1_i)_{i\in I},...,(x^n_i)_{i \in I})\mapsto(f_i(x^1_i,...,x^n_i))_{i \in I}$
\end{center}
is said to be the {\bf D-product} of the sequence  $(f_i:M^n_i \rightarrow N_i)_{i\in I}$
\end{defi}

\begin{Prop}\label{conservaci\'ondelosm\'odulosdeuniformidad}
Let ${\bf V}$ be a co-divisible value co-quantale such that $d_V$ is the symmetric distance of $\mathbf{V}$.
If $(f_i:M^n_i \rightarrow N_i)_{i\in I}$ is a sequence of uniformly continuous mappings with the same modulus of uniform continuity $\Delta:{\bf V}^+\rightarrow{\bf V}^+$, then $\Delta$ is also a modulus of uniform continuity for the $D$-product  $f_D:(M_D,d^M_D)^n\rightarrow (N_D,d^N_D)$.
\end{Prop}
\begin{proof}
For the sake of simplicity, let us take $n=1$. Let ${\bf x}=(x_i)_{i \in I}, {\bf y}=(y_i)_{i \in I}\in \prod_{i\in I}M_i$ such that $d^M_D({\bf x},{\bf y})\leq \Delta(\epsilon)$, where $\epsilon\in{\bf V}^+$ (i.e., $lim_{i,D}d_{M_i}(x_i,y_i)\leq \Delta(\epsilon)$). By Lemma~\ref{cota fuerte},  there exists $A \in D$ such that for all $i \in A$ we have that $d_{M_i}(x_i,y_i)\leq \Delta(\epsilon)$. Since $\Delta$ is a modulus of uniform continuity for $f_i$ for all $i\in I$, in particular we may say that $d_{\mathcal{N}_i}(f_i(x_i), f_i(y_i)\leq \epsilon$, whenever $i\in A$. Since $d_V$ is the symmetric distance of $\mathbf{V}$ and by Lemma~\ref{acotar D-l\'imites} (2),  $d^N_D((f_i)_{i\in I}({\bf x}),(f_i)_{i\in I}({\bf y}))\leq\epsilon$.
\end{proof}

\subsubsection{\L o\'s Theorem in value co-quantale logics.}
In order to prove a version of \L o\'s Theorem in this setting, we do not require that the involved distances of the metric structures are necessarily symmetric. Because of that, it is enough to consider the $D$-product of a sequence of $\mathbf{V}$-continuity spaces $(M_i,d_{M_i})_{i \in I}$ instead of its respective $D$-ultraproduct, like we need to do in Continuous Logic for assuring that the obtained distance is actually symmetric. However, we need to consider the respective quotient of $D$-powers of $({\bf V},d_V)$ as a $\mathbf{V}$-symmetric continuity space. From now, we assume that $({\bf V},d_V)$ is a compact, Hausdorff, co-divisible value co-quantale, where $d_V$ is the symmetric distance of $\mathbf{V}$.


\begin{defi}\label{Ultrapotencia}
Given $(M,d_M)$ a $\mathbf{V}$-continuity space where $d_M$ is symmetric, define the {\bf ultrapower} of $(M,d_M)$ as the $D$-ultraproduct\linebreak
$\left( \prod_{i\in I}M /\sim, \overline{d}_D\right)$ of the constant sequence $((M, {d}_M))_{i\in I}$
\end{defi}

\begin{defi}\label{funci\'on can\'onica al cociente}
Given a $D$-ultraproduct $M_D/\sim:=(\prod_{i\in I}M_i)/\sim$, the mapping $\theta:\prod_{i\in I}M_i\rightarrow M_D/\sim$ defined by $(x_i)_{i\in I}\mapsto [(x_i)_{i\in I}]$ is called the {\bf canonical mapping}.
\end{defi}


\begin{notation}
We denote by $({\bf V}_D/\sim, d_{V/\sim})$ the $D$-ultrapower of  the ${\bf V}$-continuity space $(V,d_V)$ (where $d_V$ is the symmetric distance).
\end{notation}

\begin{defi}\label{V-equivalencia}\index{Espacios equivalentes}
We say that two $\mathbf{V}$-continuity spaces $(X,d_X)$ and $(Y,d_Y)$ are ${\bf V}$-{\bf equivalent}, if and only if, there exists a bijection $f:(X,d_x)\rightarrow (Y,d_Y)$ such that $d_X(x,y)=d_Y(f(x),f(y))$ for any $x,y \in X$. In this case, we say that $f$ is an ${\bf V}$-equivalence.
\end{defi}

\begin{Prop}(c.f. \cite{BeBeHeUs}; pg 26)\label{ultraproductodeunespaciocompacto}
Let $(\mathbf{V},d_V)$ a value co-quantale  provided that $d_V$ is the symmetric distance, then  the $D$-ultrapower $(V_D/\sim,d_{V/\sim})$ is ${\bf V}-$equivalent to $(V,d_V)$.
\end{Prop}
\begin{proof}
Defining $T: {\bf V}\rightarrow {\bf V}/\sim$ by $x\mapsto [(x)_{i\in I}]$, we have that $T$ is a ${\bf V}$-equivalence. In fact,  $T$ is injective:  Let $x,y\in{\bf V}$ be such that $T(x)=T(y)$, so by definition of $\sim$ (Proposition~ \ref{paso al cociente}) we have that $d_{V/\sim}(T(x),T(y))=d_{V/\sim}([(x)_{i\in I}],[(y)_{i\in I}])=\lim_{i,D} d_V(x,y)=d_V(x,y)=(x\dotdiv y)\vee(y\dotdiv x)=0$, therefore $x\dotdiv y=0$ and $y\dotdiv x=0$. By Proposition~\ref{adjunci\'on} (4) we may say that $x\leq y$ and $y\leq x$, and so $x=y$.

In order to prove that $T$ is surjective, let $[(x_i)_{i\in I}]\in V_D$. Let us see that $T(lim_{i,D}x_i):=[(lim_{i,D}x_i)_{i\in I}]=[(x_i)_{i\in I}]$ (i.e., we will have that $lim_{i,D}d_V(lim_{i.D}x_i,x_i)=0$)): let $\epsilon\in{\bf V}^+$, 
therefore  $d_V(0,d_V(lim_{i.D}x_i,x_i))=\vee\{0\dotdiv\ d_V(lim_{i.D}x_i,x_i),d_V(lim_{i.D}x_i,x_i)\dotdiv 0\}=\vee\{0,d_V(lim_{i.D}x_i,x_i)\}=d_V(lim_{i.D}x_i,x_i)$, and by definition of $lim_{i.D}x_i$ we may say that $\{i\in I : d_V(lim_{i.D}x_i,x_i) \leq \epsilon \}\in D$. Hence, $\{i\in I : d_V(0,d_V(lim_{i.D}x_i,x_i)) \leq \epsilon \}\in D$ and then  $lim_{i,D}d_V(lim_{i.D}x_i,x_i)=0$.

$T$ preserves distances: In fact, $d_{V_D}(T(x),T(y))=d_{V_D}([(x)_{i\in I}],[(y)_{i\in I}]):=\lim_{i,D}(d_V(x,y))_{i\in I}=d_V(x,y)$.

Therefore, $T$ is an equivalence.
\end{proof}

\begin{remark}\label{Inversa de la diagonal}
Notice that the mapping $T':{\bf V}_D/\sim \ \to {\bf V}$ defined by $[(x_i)_{i\in I}]\mapsto lim_{i,D}x_i$ is the inverse of $T$.
\end{remark}

\begin{hec}\label{m\'odulos sobre el cocinte}
Given a sequence of $\mathbf{V}$-continuity spaces $((M_i,d_{M_i}))_{i\in I}$ provided that all distances $d_{M_i}$ are symmetric, the canonical mapping $\theta:\prod_{i\in I}M_i\rightarrow M_D/\sim$ is uniformly continuous with modulus of uniform continuity  $id_{{\bf V}^+}$.
\end{hec}

\begin{defi}\label{interpretaci\'on en el D-producto}
Suppose that ${\bf V}$ is co-divisible and let ${\bf L}$ be a language based in a continuous structure. Given a sequence $(\mathcal{M}_i)_{i\in I}$ of ${\bf L}$-structures, define the {\bf D-product} of $(\mathcal{M}_i)_{i\in I}$ as the ${\bf L}$-structure $\mathcal{M}_D$ with underlying $\mathbf{V}$-continuity space $(M_D,d_D)$, defined as follows:
\begin{enumerate}
    \item For a predicate symbol $R\in{\bf L}$, define $R^{\prod_{i\in I}\mathcal{M}_i}:=T^\prime\circ\theta\circ R_D$, where $R_D$ is the $D$-product of the mappings $(R^{\mathcal{M}_i})_{i\in I}$, $\theta$ the canonical mapping given in Definition~\ref{funci\'on can\'onica al cociente} and $T^\prime$ the mapping defined in Remark~\ref{Inversa de la diagonal}. 
    \item For a function symbol $F\in {\bf L}$, defined $F^{\prod_{i\in I}\mathcal{M}_i}$ as the $D$-product of the mappings $(F^{\mathcal{M}_i})_{i\in I}$.
    \item For a constant symbol $c\in {\bf L}$, define $c^{\prod_{i\in I}\mathcal{M}_i}:=(c^{\mathcal{M}_i})_{i\in I}$.
\end{enumerate}
\end{defi}

\begin{remark}\label{nota importante}
Notice that Theorem~\ref{conservaci\'ondelosm\'odulosdeuniformidad} guarantees that the interpretations of the symbols of $\mathbf{L}$ given above have the same modulus of uniform continuity given by the language.
\end{remark}

\begin{theo}(\L o\'s Theorem; c.f. \cite{BeBeHeUs} Thrm 5.4)\label{Teorema de Los}. Let $({\bf V},d_V)$ be a co-divisible ${\bf V}$-domain. If $(\mathcal{M}_i)_{i\in I}$ is a sequence of ${\bf L}$-structures, then for any  ${\bf L}$-formula $\phi(x_1,...,x_n)$ (if $\phi$ has quantifiers, we require that its interpretations in any ${\bf L}$-structure $\mathcal{M}_i$ satisfy the hypothesis of Proposition~\ref{ser de Cauchy discreto}) and any tuple $((a^1_i)_{i \in I},...,(a^n_i)_{i \in I})\in(\prod_{i\in I}M_i)^n$, we have that
\begin{center}
    $\phi^{\mathcal{M_D}} ((a^1_i)_{i \in I},...,(a^n_i)_{i \in I})=lim_{i,D}\phi^{\mathcal{M}_i}(a^1_i,...,a^n_i)$
\end{center}
\end{theo}
\begin{proof}
We proceed by induction on ${\bf L}$-formulae.\\

\begin{enumerate}
\item $\phi : d(x_1,x_2)$
\begin{eqnarray*}
(d(x_1,x_2))^{\mathcal{M}}((a^1_i)_{i \in I},(a^2_i)_{i \in I})&:=&d^{\mathcal{M}}((a^1_i)_{i \in I},(a^2_i)_{i \in I}) \\ 
      &=&lim_{i, D}d^{\mathcal{M}_i}(a^1_i,a^2_i) \\ && (\text{definition of a distance in a product} )\\
      &=&lim_{i, D}(d(x_1,x_2))^{\mathcal{M}_i}(a^1_i,a^2_i)) \\ 
\end{eqnarray*}

\item $\phi: R(x_1,...,x_n)$, where $R$ is a predicate symbol in ${\bf L}$.
\begin{eqnarray*}
(R(x_1,...,x_n)((a^1_i)_{i \in I},...,(a^n_i)_{i \in I}))^{\mathcal{M}}&:=&(R((a^1_i)_{i \in I},...,(a^n_i)_{i \in I}))^{\mathcal{M}} \\
&=&T^\prime\circ\theta((R^{\mathcal{M}_i}(a^1_i,...,a^n_i))_{i\in I})\\
&&(\text{by Definition~\ref{interpretaci\'on en el D-producto} (1)})\\
&=&lim_{i,D}R^{\mathcal{M}_i}((a^1_i,...,a^n_i))\\
&&(\text{by definition of $\theta$ and $T^\prime$})\\
\end{eqnarray*}

\item $\phi: u(\sigma_1,...,\sigma_m)(x_1,...,x_n)$, where $u:\mathbf{V}^m\to \mathbf{V}$ is a uniformly continuous mapping and $\sigma_1,\cdots \sigma_m$ are $\mathbf{L}$-formulae such that $\sigma_k^{\mathcal{M}}((a^1_i)_{i \in I},...,(a^n_i)_{i \in I})=lim_{i,D}\sigma_k^{\mathcal{M}_i}(a^1_i,...,a^n_i)$ for all $k\in\{1,...,m\}$ (induction hypothesis). For the sake of simplicity, denote\linebreak $((a^1_i)_{i \in I},...,(a^n_i)_{i \in I})=:\overline{a}$ and $(a^1_i,...,a^n_i)=\overline{a}_i$.
{
\begin{eqnarray*}
(u(\sigma_1,...\sigma_m))^{\mathcal{M}}(\overline{a})&:=&u(\sigma^{\mathcal{M}}_1(\overline{a}),...,\sigma^{\mathcal{M}}_m(\overline{a})) \\
&=& u(lim_{i,D}\sigma^{\mathcal{M}_i}_1(\overline{a}_i),...,lim_{i,D}\sigma^{\mathcal{M}_i}_m(\overline{a}_i) ) \\ &&(\text{induction hypothesis})\\
\end{eqnarray*}
}
Define $b_{i,k}:=\sigma^{\mathcal{M}_i}_k(\overline{a}_i)$ for any $k\in \{1,\cdots,m\}$ and $i\in I$.
Notice that $$\{i \in I : d_{V}(u(lim_{i,D}b_{i,1},...,lim_{i,D}b_{i,m}), u(b_{i,1},...,b_{i,m}))\leq \epsilon \}$$ contains the set
$$\{ i \in I: d_{V^n}((lim_{i,D}b_{i,1}...,lim_{i,D}b_{i,m}),(b_{i,1},...,b_{i,m})) \leq \Delta(\epsilon) \}$$
$$= \left\{ i \in I : \bigvee^n_{k=1} \{d_{V}(lim_{i,D}b_{i,k},b_{i,k} )\} \leq \Delta(\epsilon) \right\},$$ whenever $\Delta$ is a modulus of uniform continuity for $u$.\\ Notice that this previous set belongs to $D$, because it contains $\bigcap_{k=1}^n \left\{i \in I | d_{V}(lim_{i,D}b_{i,k},b_{i,k} ) \leq \Delta(\epsilon) \right\}$, which belongs to $D$ by definition of a $D$-limit and since $D$ is an ultrafilter.

Therefore, 
\begin{eqnarray*}
u(lim_{i,D}\sigma^{\mathcal{M}_i}_1(\overline{a}_i),...,lim_{i,D}\sigma^{\mathcal{M}_i}_m(\overline{a}_i)) &=& lim_{i,D} u(\sigma^{\mathcal{M}_i}_1(\overline{a}_i),\cdots \sigma^{\mathcal{M}_i}_m(\overline{a}_i))\\
&& lim_{i,D} \left( u(\sigma_1,\cdots, \sigma_m) \right)^{\mathcal{M}_i}(\overline{a}_i).
\end{eqnarray*}

\item $\phi: \bigvee_x \varphi(x,x_1,...,x_n)$\\
Let $\varphi(x,x_1,...,x_n)$ be an ${\bf L}$-formula such that $$\varphi^{\mathcal{M}}((b_i)_{i\in I},(a^1_i)_{i\in I},...,(a^n_i)_{i\in I})=lim_{i\in I}\varphi^{\mathcal{M}_i}(b_i,a^1_i,...,a^n_i)$$
for all $(b_i)_{i\in I}\in\prod_{i\in I}M_i$ (induction hypothesis). For the sake of simplicity, denote $\overline{a}:=((a^1_i)_{i \in I},...,(a^n_i)_{i \in I})$, $\overline{a}_i:=(a^1_i,...,a^n_i)$ and $b:=(b_i)_{i\in I}$. So, 

\begin{eqnarray*}
\left(\bigvee_x\varphi(x,x_1,...,x_n)\right)^{\mathcal{M}}(\overline{a})&=&\left(\bigvee_x\varphi(x,\overline{a})\right)^{\mathcal{M}} \\
&=&\bigvee\{\varphi^{\mathcal{M}}(b,\overline{a}) :b\in \mathcal{M}\} \\ &&  \text{(by Definition~\ref{interpretaci\'on de sentencias} (4)) } \\
&=&\bigvee\{lim_{i,D}\varphi^{\mathcal{M}_i}(b_i,\overline{a}_i): b\in \mathcal{M} \} \\ &&(\text{induction hypothesis})\\
&\leq& lim_{i,D}\bigvee\{\varphi^{\mathcal{M}_i}(b_i,\overline{a_i}): b\in \mathcal{M} \} \\  &&(\text{by Proposition~\ref{teorema de cuantificadores}})
\end{eqnarray*}

By Proposition~\ref{teorema de cuantificadores}, given $\epsilon \in{\bf V}^+$ there exists a sequence $(\overline{c^j})_{j\in I}$ of tuples $\overline{c^j}:=(c^j_i)_{i\in I}\in \prod_{i\in I} M_i$ such that 
\begin{eqnarray*}
lim_{i,D}(\bigvee\{\varphi^{\mathcal{M}_i}(b_i,\overline{a}_i): b\in \mathcal{M} \})&\leq& lim_{i,D}\varphi^{\mathcal{M}_i}(c^i_i, \overline{a}_i) + \epsilon\\
&\leq& \bigvee\{lim_{i,D}\varphi^{\mathcal{M}_i}(b_i,\overline{a}(i)): b\in \mathcal{M} \}+\epsilon
\end{eqnarray*}

where the last inequality follows from monotonicity of $\le$ and since $(c^i_i)_{i\in I}\in\prod_{i\in I}M_i$. Since $\epsilon \in {\bf V}^+$ is arbitrary, we have that  $lim_{i,D}\bigvee\{\varphi^{\mathcal{M}_i}(b_i,\overline{a}_i):b\in \mathcal{M}\}\leq lim_{i,D}\bigvee\{\varphi^{\mathcal{M}_i}(b_i,\overline{a}_i): b\in \mathcal{M} \}$. Therefore, by antisymmetry of $\le$ we may say that  $lim_{i,D}\bigvee\{\varphi^{\mathcal{M}_i}(b(i),\overline{a_i}): b\in \mathcal{M} \}=\bigvee\{lim_{i,D}\varphi^{\mathcal{M}_i}(b(i),\overline{a}(i)): b\in \mathcal{M} \}$. Notice that $\bigvee\{\varphi^{\mathcal{M}_i}(b_i,\overline{a}_i):b\in \mathcal{M}\}=\bigvee\{\varphi(c,\overline{a}_i):c\in \mathcal{M}_i\}$, then $(\bigvee_x\varphi(x,x_1,...,x_n))^{\mathcal{M}}(\overline{a})=lim_{i,D}\bigvee\{\varphi(c,\overline{a}_i):c\in \mathcal{M}_i\}=lim_{i,D}(\bigvee_x\varphi(x,x_1,...,x_n))^{\mathcal{M}_i}(\overline{a}_i)$, as desired.\\
\item $\phi: \bigwedge_x \varphi(x,x_1,...,x_n)$\\
Let $\varphi(x,x_1,...,x_n)$ be an ${\bf L}$-formula such that $$\varphi^{\mathcal{M}}((b_i)_{i\in I},(a^1_i)_{i\in I},...,(a^n_i)_{i\in I})=lim_{i\in I}\varphi^{\mathcal{M}_i}(b_i,a^1_i,...,a^n_i)$$ for all $(b_i)_{i\in I}\in\prod_{i\in I}M_i$. For the sake of simplicity, denote $\overline{a}:=((a^1_i)_{i \in I},...,(a^n_i)_{i \in I})$, $\overline{a}_i:=(a^1_i,...,a^n_i)$ and $b:=(b_i)_{i\in I}$. So, 
\begin{eqnarray*}
\left(\bigwedge_x\varphi(x,x_1,...,x_n)\right)^{\mathcal{M}}(\overline{a})&=&\left(\bigwedge_x\varphi(x,\overline{a})\right)^{\mathcal{M}}\\
&=&\bigwedge\{\varphi^{\mathcal{M}}(b,\overline{a}) :b\in \mathcal{M}\} \\ &&(\text{by Definition~\ref{interpretaci\'on de sentencias} (5)}) \\
&=&\bigwedge\{lim_{i,D}\varphi^{\mathcal{M}_i}(b_i,\overline{a}_i): b\in \mathcal{M} \} \\&&( \text{induction hypothesis})\\
&\geq& lim_{i,D}\bigwedge\{\varphi^{\mathcal{M}_i}(b_i,\overline{a}_i): b\in \mathcal{M} \} \\ &&( \text{by Proposition~\ref{teorema de cuantificadores}})
\end{eqnarray*}

By Proposition~\ref{teorema de cuantificadores} and by hypothesis, given $\epsilon \in{\bf V}^+$ there exists a sequence $(\overline{b^j})_{j\in I}$ of tuples $\overline{b^j}:=(b^j_i)_{i\in I}\in \prod_{i\in I} M_i$ such that 

\begin{eqnarray*}
lim_{i,D}(\bigwedge\{\varphi^{\mathcal{M}_i}(b_i,\overline{a}_i): b\in \mathcal{M} \})+\epsilon&\geq& lim_{i,D}\varphi^{\mathcal{M}_i}(b^i_i, \overline{a}_i)\\
&\geq& \bigwedge\{lim_{i,D}\varphi^{\mathcal{M}_i}(b_i,\overline{a}_i): b\in \mathcal{M} \}
\end{eqnarray*}

where the last inequality follows from the fact that $(b^i_i)_{i\in I}\in\prod_{i\in I}M_i$. Since $\epsilon \in {\bf V}^+$ was taken arbitrarily, we have that
$lim_{i,D}\bigwedge\{\varphi^{\mathcal{M}_i}(b_i,\overline{a}_i):b\in \mathcal{M}\}\geq\bigwedge\{lim_{i,D}\varphi^{\mathcal{M}_i}(b_i,\overline{a}_i): b\in \mathcal{M} \}\geq lim_{i,D}\bigwedge\{\varphi^{\mathcal{M}_i}(b(i),\overline{a_i}): b\in \mathcal{M} \}$. By antisymmetry of $\le$, we may say that $$lim_{i,D}\bigwedge\{\varphi^{\mathcal{M}_i}(b_i,\overline{a_i}): b\in \mathcal{M} \}=\bigwedge\{lim_{i,D}\varphi^{\mathcal{M}_i}(b_i,\overline{a}_i): b\in \mathcal{M} \}.$$ Since, $\bigwedge\{\varphi^{\mathcal{M}_i}(b_i,\overline{a}_i):b\in \mathcal{M}\}=\bigwedge\{\varphi^{\mathcal{M}_i}(c,\overline{a}_i):c\in \mathcal{M}_i\}$, then $(\bigwedge_x\varphi(x,x_1,...,x_n))^{\mathcal{M}}(\overline{a})=lim_{i,D}\bigwedge\{\varphi(c,\overline{a}_i):c\in \mathcal{M}_i\}=lim_{i,D}(\bigwedge_x\varphi(x,x_1,...,x_n))^{\mathcal{M}_i}(\overline{a}_i)$, as desired.
\end{enumerate}
\end{proof}

\begin{nota}
In case that we want to work in symmetric spaces, the same argument as above works  to prove a version of \L o\'s Theorem for the $D$-ultraproduct $\left(\prod_{i\in I}\mathcal{M}_i\right)_{\sim}$.
\end{nota}

\subsubsection{Some consequences of \L o\'s Theorem.}

In this setting, \L o\'s Theorem implies a version of Compactness Theorem and the existence of some kind of $\omega_1$-saturated models, as it holds in both first order and Continuous logics.

First, we provide a proof of Compactness Theorem, up to \L o\'s Theorem.

\begin{Coro}(Compactness Theorem, c.f. \cite{BeBeHeUs} Thrm 5.8)\label{teorema de compacidad}
Let ${\bf L}$ be a language based on a continuous structure. Let $\textit{T}$ be an  ${\bf L}-$ theory which conditions satisfy the hypothesis of Theorem~\ref{Teorema de Los} and $\mathcal{C}$ be a  class of ${\bf L}$-structures. Therefore, if $\textit{T}$ is finitary satisfiable in $\mathcal{C}$, then there exists a $D$-product of structures in $\mathcal{C}$ that is a model of $\textit{T}$.
\end{Coro}
\begin{proof}
Let $\Lambda$ be the collection of all finite subsets of $\textit{T}$. By hypothesis, given $\lambda\in\Lambda$ with $\lambda:=\{E_1,...,E_n\}$, there exists some $\mathcal{M}_\lambda\in\mathcal{C}$ such that $\mathcal{M}_\lambda\models E_k$ for all $k\in\{1,...,n\}$.

Fixed an ${\bf L}$-condition $E\in\textit{T}$, define $S(E):=\{\lambda\in\Lambda : E\in\lambda\}$. Notice that $S(E_1)\cap ... \cap S(E_n)\neq \emptyset$ ( $\lambda:=\{E_1,\cdots,E_n\}\in S(E_1)\cap ...\cap S(E_n)$), therefore $\{S(E) : E\in \textit{T}\}$ satisfies the Finite Intersection Property. Let $D$ be an ultrafilter over $\Lambda$ extending $\{S(E) : E\in \textit{T}\}$. Let  $\mathcal{M}:=\prod_{\lambda\in \Lambda}\mathcal{M}_\lambda$ be the respective $D$-product of the sequence of $\mathbf{L}$-structures $(\mathcal{M}_\lambda)_{\lambda\in \Lambda}$. Given $E\in \textit{T}$, where $E:\psi=0$ ($\psi$ an ${\bf L}$-sentence). Notice that for all $\lambda\in S(E)$ we have that $\mathcal{M}_\lambda\models E$; i.e., $\psi^{\mathcal{M}_\lambda}=0$. Since by construction $S(E)\in D$, by \L o\'s Theorem (Theorem~\ref{Teorema de  Los}) it follows that $lim_{\lambda,D}\psi^{\mathcal{M}_\lambda}=0$. Notice that $\psi^{\mathcal{M}}=lim_{\lambda,D}\psi^{\mathcal{M}_\lambda}=0$, so $\mathcal{M}\models E$. Therefore, $\mathcal{M}\models\textit{T}$.
\end{proof}

Keisler showed in \cite{Keisler1964} the existence of saturated structures by using ultraproducts. In \cite{BeBeHeUs}, there is a proof of an analogous result to Keisler's construction  by using metric ultraproducts. In the following lines, we provide a proof of this result in the logic propposed in this paper, supposing that $\mathbf{V}$ is provided with the symmetric distance (abusing of the notation, we will denote by $d_V$) and  co-divisible value co-quantale satistying the SAFA Property (Definition~\ref{propiedad SAFA}).

\begin{defi}
Let ${\bf L}$ be a language based on a continuous structure, $\Gamma(x_1,...,x_n)$ be a set of ${\bf L}$-conditions and $\mathcal{M}$ be an ${\bf L}$-structure. We say that $\Gamma(x_1,...,x_n)$ is satisfiable in $\mathcal{M}$, if and only if, there exist $a_1,...,a_n\in M$ such that $\mathcal{M}\models E(a_1,...,a_n)$ for all $E(x_1,\cdots, x_n)\in \Gamma(x_1,...,x_n)$.
\end{defi}

\begin{defi}
Let ${\bf L}$ be a language based on a continuous structure , $\mathcal{M}$  be an ${\bf L}$-structure and $\kappa$ be an infinite cardinal. We say that $\mathcal{M}$ is {\bf $\kappa$-saturated}, if and only if, given $A\subseteq M$ such that $|A|<\kappa$ and $\Gamma(x_1,...,x_n)$ a set of ${\bf L}(A)$-conditions with parameters in $A$, it holds that if $\Gamma(x_1,...,x_n)$ is finitely satisfiable in $\mathcal{M}$ then $\Gamma(x_1,...,x_n)$ is satisfiable in $\mathcal{M}$. 
\end{defi}


\begin{defi}
An ultrafilter $D$ over $K\neq \emptyset$ is said to be {\bf countably-incomplete}, if and only if, there exists some $\{A_n: n\in \mathrm{N}\}\subseteq D$ such that $\bigcap_{n\in \mathrm{N}}A_n=\emptyset$.
\end{defi}

\begin{Prop}\label{caracterizaci\'on contable-incompleto}
Let $D$ be an coutably-incomplete ultrafilter over $K\neq\emptyset$, then there exists a countable subcollection $\{J_n: n\in \mathbb{N}\}$ of $D$ such that $J_{n+1}\subseteq J_n$ for all $n\in\mathbb{N}$ and $\bigcap_{n\in \mathrm{N}}J_n=\emptyset$.
\end{Prop}
\begin{proof}
By hypothesis, there exists a subcollection $\{A_n: n\in \mathrm{N}\}\subseteq D$ of $D$ tal que $\bigcap_{n\in \mathrm{N}}A_n=\emptyset$. Define $J_0:=A_0$ and $J_{n+1}:=J_n\cap A_{n+1}$ for any $n\in \mathbb{N}\setminus\{0\}$.
\end{proof}

\begin{Prop}\label{existencia de modelos saturados}(c.f. \cite{BeBeHeUs}; Prop. 7.6)
Let ${\bf V}$ be a compact, Hausdorff, co-divisible value co-quantale satisfying SAFA, provided with the symmetric distance $d_V$. Let ${\bf L}$ be a countable language based on a continuous structure and $D$ be a countably-incomplete ultrafilter over a non empty set $\Lambda$. Given any $\Lambda$-sequence of ${\bf L}$-structures $(\mathcal{M}_\lambda)_{\lambda\in\Lambda}$, its $D$-product $M_D$ is $\omega_1$-saturated, assuming that all $\mathbf{L}$-formulae satisfy the hypothesis in \L o\'s Theorem (Theorem~ \ref{Teorema de Los}). 
\end{Prop}
\begin{proof}
For the sake of simplicity, let us analyze  ${\bf L}$-conditions with one variable $x$. Let $A\subseteq \prod_{\lambda\in\Lambda}M_{\lambda}$ be countable and $\Gamma(x)$ be a set of ${\bf L}(A)$-conditions with parameters in $A$ which is finitely satisfiable in $\mathcal{M}_D$. We will prove that $\Gamma
(x)$ is satisfiable in $\mathcal{M}_D$. 

Since ${\bf L}$ is countable, let $\Gamma(x):=\{\psi_n(x) : n<\omega\}$ be an enumeration of $\Gamma(x)$. Since $D$ is contably-incomplete, by Proposition~\ref{caracterizaci\'on contable-incompleto} there exists a sequence  $(J_n)_{n\in\mathbb{N}}$ of elements in $D$ such that $J_{n+1}\subseteq J_n$ for all $n\in \mathbb{N}$ and $\bigcap_{n\in\mathbb{N}}J_n=\emptyset$.

By hypothesis, for all $k\in\mathbb{N}$ the set $\{\psi_1(x),...,\psi_k(x)\}$ is satisfiable in $\prod_{\lambda\in\Lambda}\mathcal{M}_{\lambda}$. By \L o\'s Theorem (Theorem~\ref{Teorema de Los}), there exists some $a:=(a_\lambda)_{\lambda\in \Lambda}\in \prod_{\lambda\in\Lambda}M_{\lambda}$ such that for all $n\in \{1,\cdots,k\}$ we have that $\psi_n^{M_D}\mathcal{M}_{\lambda}(a)=lim_{\lambda,D}\psi^{\mathcal{M}_\lambda}_n(a_\lambda)=0$. Therefore, given $\epsilon\in{\bf V}^+$ we may say that $\{\lambda\in\Lambda : d_V(0,\psi^{\mathcal{M}_\lambda}_n(a_\lambda))\leq \epsilon\}\in D$. Since $d_V$ is the symmetric distance, by Remark~\ref{la distancia al 0 en el caso sim\'etrico} we have that $\{\lambda\in\Lambda : \psi^{\mathcal{M}_\lambda}_n(a_\lambda)\leq\epsilon\}\in D$. Since ${\bf V}$ satisfies SAFA Property (Definition~\ref{propiedad SAFA}), there exists a sequence $(u_k)_{k\in\mathbb{N}}$ in $\mathbf{V}$ such that
\begin{enumerate}
    \item $\bigwedge_{n\in\mathbb{N}}u_n=0.$
    \item for all $n\in\mathbb{N}$, $0\prec u_n$
    \item for all $n\in\mathbb{N}$, $u_{n+1}\leq u_n$
\end{enumerate}
Therefore, 
$\{\lambda\in\Lambda : \psi^{\mathcal{M}_\lambda}_n(a_\lambda)\leq u_l\}\in D$ for all $l\in \mathbb{N}$. This implies that $A_k:=\{\lambda\in \Lambda : \mathcal{M}_{\lambda}\models \bigwedge_{x\in M_{\lambda}}\bigvee^k_{n=1}\psi_n(x)\leq u_{k+1}\}\in D$ whenever $k\in\mathbb{N}$. Define the sequence  $(X_n)_{n\in\mathbb{N}}$ of elements in $D$ as follows: $X_0:=\Lambda$ and $X_k:=J_k\cap A_k$ if $k\in\mathbb{N}\setminus\{0\}$. Notice that $\bigwedge_{x\in M_\lambda}\bigvee_{n=1}^k \psi_n(x)\le \bigvee_{x\in M_\lambda}\bigwedge_{n=1}^{k+1} \psi_n(x)\le u_{k+2}\le u_{k+1}$, then $A_{n+1}\subseteq A_n$ for all $n\in \mathbb{N}$ and then $X_{n+1}\subseteq X_n$ for all $n\in \mathbb{N}$. Notice that $\bigcap_{n\in\mathbb{N}}X_n=\emptyset$, therefore if $\lambda\in \Lambda$ there exists $k_\lambda\in\mathbb{N}$ such that $k_\lambda:=max\{n\in\mathbb{N}: \lambda\in X_n\}$. Define $\overline{a}:=(a_\lambda)_{\lambda\in\Lambda}\in \prod_{\lambda\in\Lambda}\mathcal{M}_{\lambda}$ as follows: In case that $k_\lambda=0$ , take $a_\lambda$ as any element in $M_{\lambda}$, otherwise take $a_\lambda\in M_\lambda$ such that $\bigvee\{\psi^{\mathcal{M}_\lambda}_n(a_\lambda) : n\leq k_\lambda \}\leq u_{k_\lambda}$. So, if $k\in\mathbb{N}$ then for any $n\in\mathbb{N}$ such that $k\leq n$ and  $\lambda\in X_n$ we have that $n\leq k_\lambda$, hence $\psi^{\mathcal{M}_\lambda}_k(a_\lambda)\leq u_{k_\lambda}\leq u_n$. Since $X_n\in D$, by \L o\'s Theorem (Theorem~\ref{Teorema de Los}) we have that  $\psi_k^{\prod_{\lambda\in\Lambda}\mathcal{M}_\lambda}(\overline{a})=lim_{\lambda,D}\psi^{M_D}_k(a_\lambda)=0$. Since $\psi_k(x)\in\Gamma(x)$ was taken arbitrarily, then $\overline{a}\in \prod_{\lambda\in\Lambda}\mathcal{M}_\lambda$ realizes $\Gamma(x)$, as desired.
\end{proof}

\section{Characterization of Continuous Logic and new logics}
J. Iovino proved in~\cite{Io01} that there is no any logic 
extending properly (logics equivalent to) Continuous Logic satisfying both\linebreak
Countable Tarski-Vaught chain Theorem and Compactness Theorem. Therefore, assuming that a value co-quantale is co-Girard, co-divisible (substractable) and a $V$-domain, we already proved that the logic given in Section~\ref{VLogics} satisfies the Tarski-Vaught test (Proposition~\ref{test de Tarski-Vaught}) and \L o\'s Theorem (Theorem~\ref{Teorema de Los}); therefore this logic satisfies the Countable Tarski-Vaught chain Theorem and Compactness Theorem (Theorem~\ref{teorema de compacidad}).
\ \\ \\
\indent 
It is straightforward to see that $\mathbb{V}:=([0,1],0,+,\le)$ is co-Girard, co-divisible and a $V$-domain. In this way, we got a characterization of Continous Logic by using the approach studied in this paper.
\ \\ \\
\indent Therefore, if we drop any of these assumptions, it suggests that we get new logics in the way proposed in Section~\ref{VLogics}.
\indent For example, notice that $[0,\infty]$ is not a $[0,\infty]$-domain (Remark 4.11 and Theorem 4.8, \cite{FlaKop97}). Hence, $[0,\infty]$ yields a new logic different to Continuous Logic. 
\bibliographystyle{alpha}
\bibliography{main2}
\end{document}